\documentclass[oneside,11pt]{amsart}

\usepackage[margin=1in]{geometry}

\usepackage{graphicx}
\usepackage{color}
\usepackage{amsthm,amsmath,amssymb,mathrsfs,eucal}
\usepackage[ruled,vlined]{algorithm2e}

\newtheorem{thm}{Theorem}[section]
\newtheorem{cor}[thm]{Corollary}
\newtheorem{lem}[thm]{Lemma}
\newtheorem{prop}[thm]{Proposition}
\theoremstyle{definition}

\newtheorem{aspt}{Assumption}
\theoremstyle{remark}
\newtheorem{rem}[thm]{Remark}

\makeatletter
\def\@captionheadfont{\scshape\small}
\def\@captionfont{\normalfont\small}
\makeatother

\begin{document}

\title[Uniform Time Average Consistency of Monte Carlo Particle 
Filters]{Uniform Time Average Consistency \\ of Monte Carlo Particle 
Filters}

\author{Ramon van Handel}
\thanks{The author is partially supported by the NSF RTG Grant DMS-0739195.}
\address{Department of Operations Research and Financial Engineering,
Princeton University, Princeton, NJ 08544}
\email{rvan@princeton.edu}



\begin{abstract}
We prove that bootstrap type Monte Carlo particle filters approximate the 
optimal nonlinear filter in a time average sense uniformly with respect 
to the time horizon when the signal is ergodic and the particle system 
satisfies a tightness property.  The latter is satisfied without further 
assumptions when the signal state space is compact, as well as in the 
noncompact setting when the signal is geometrically ergodic and the 
observations satisfy additional regularity assumptions.
\end{abstract}

\maketitle

\section{Introduction}

Consider a hidden Markov model of the form
$$
	X_n = f(X_{n-1},\xi_n),\qquad\qquad
	Y_n = h(X_n,\eta_n),
$$
where $(\xi_n)_{n\ge 1}$, $(\eta_n)_{n\ge 0}$ are independent i.i.d.\ 
sequences.  The signal $X_n$ represents a dynamical process of interest, 
but only the noisy observations $Y_n$ are available.  More 
generally, $(X_n)_{n\ge 0}$ may be any Markov process and $(Y_n)_{n\ge 0}$
are assumed to be conditionally independent given $(X_n)_{n\ge 0}$.  Such 
models appear in a wide variety of applications (see, e.g., \cite{DDG01}). 
As the signal is not directly observed, one is generally faced with the 
problem of estimating the signal on the basis of the observations.  
To this end, the \emph{nonlinear filtering} problem aims to compute the 
conditional distribution $\pi_n$ of the signal $X_n$ given the observation 
history $Y_0,\ldots,Y_n$ in a recursive (on-line) fashion.

The theory of nonlinear filtering is a classic topic in probability 
\cite{LS77} and statistics \cite{BPSW70}.  Unfortunately, the theory 
suffers in practice from the fact that the conditional distribution 
$\pi_n$ is an infinite dimensional object.  With the exception of some 
special cases, the filtering recursion can not be represented in a finite 
dimensional fashion and its direct implementation is therefore 
intractable. For this reason, realistic applications have long remained 
limited.

This state of affairs was revolutionized in the early 1990s by 
the discovery \cite{GSS93} of a new class of approximate nonlinear 
filtering algorithms based on Monte Carlo simulation, which are known 
under various names in the literature: bootstrap filters, interacting 
particle filters, sequential Monte Carlo filters, etc.  Such algorithms 
are simple to implement (even for complex models), are computationally 
tractable, typically exhibit excellent performance, and can be rigorously 
proved to converge to the exact nonlinear filter when the number of 
samples is large.  These techniques have consequently been applied in 
problems ranging from robotics to finance, and their theoretical 
properties have been investigated by many authors; we refer to the 
collection \cite{DDG01} for a general introduction to the theory and 
applications of Monte Carlo particle filters, while a detailed overview of 
theoretical developments can be found in the recent monographs
\cite{Del04,CMR05}.

Despite many advances in recent years, however, certain empirically 
observed properties of Monte Carlo particle filters remain poorly understood 
theoretically.  The aim of this paper is to study one such property:
the \emph{uniform} nature of the particle filter approximation.  

\subsection{A toy example}

The uniform nature of particle filter approximations is most easily 
illustrated by means of a simple but illuminating numerical example.
Let us consider the filtering model 
$$
	X_n = 0.9\,X_{n-1}+\xi_n,\quad
	X_0=0,\qquad\qquad
	Y_n = X_n + \eta_n,
$$
where $\xi_n,\eta_n$ are i.i.d.\ $N(0,1)$.  As only the observations are 
available to us, we aim to compute the conditional mean of the signal 
$\mathbf{E}(X_n|Y_0,\ldots,Y_n)$.  In this very special case, it is well 
known that the latter can be computed exactly using a finite dimensional 
algorithm (the Kalman filter).  

\begin{figure}
\centering
\vskip-.8cm
\includegraphics[width=.75\textwidth]{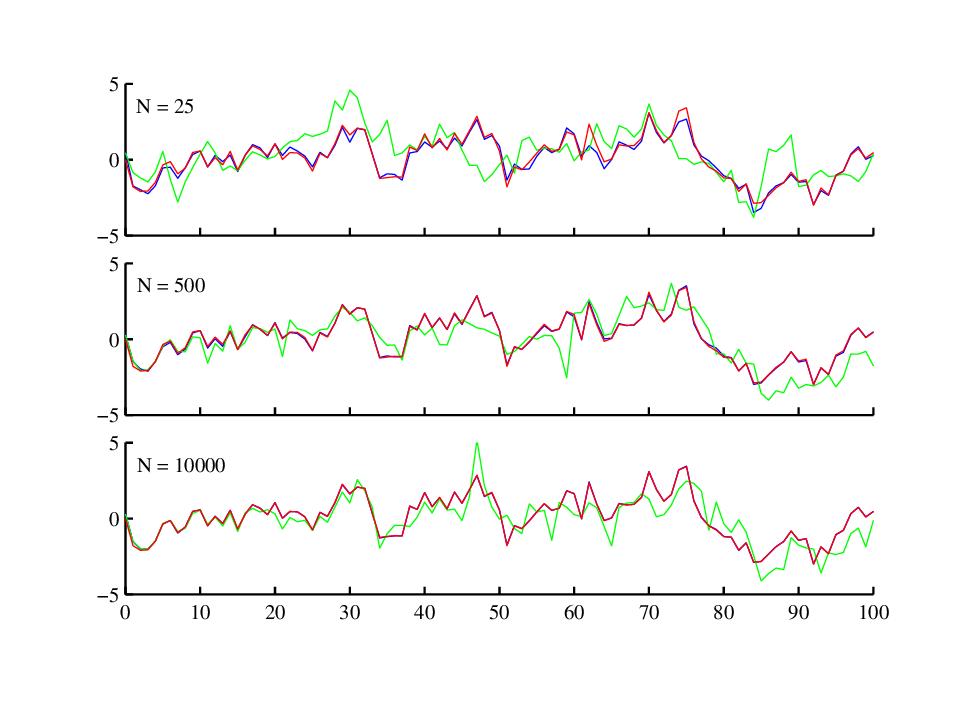}
\vskip-1cm
\caption{The conditional mean $\mathbf{E}(X_n|Y_0,\ldots,Y_n)$ (red) and 
approximations by the bootstrap (blue) and naive (green) particle filter 
for a single sample path of the model described in the text.  The number 
of particles $N$ used for the approximate filters varies in each plot.
\label{fig:sisr}}
\end{figure}

A numerical simulation of this example is shown in figure \ref{fig:sisr}, 
where we have plotted the exact conditional mean and its approximation 
obtained by means of the bootstrap particle filter.  For sake of 
illustration, we have plotted also a different `naive' Monte Carlo 
approximation of the conditional mean which, like the bootstrap filter, is 
easily proved to converge to the exact conditional mean when the number of 
Monte Carlo particles is large.  [The precise details of these algorithms 
will be given in section \ref{sec:bootstrap} below, and are irrelevant to 
the present discussion.] Though both algorithms converge, the difference 
in performance between the two algorithms is striking: the approximation 
error of the naive algorithm grows rapidly in time, while the error of the 
bootstrap algorithm appears to be independent of time (see \cite{DJ01} for 
further computations in this example).

Evidently the fact that both algorithms converge does not capture the 
key qualitative advantage of the bootstrap filter over the naive 
algorithm: the bootstrap filter converges to the exact filter 
\emph{uniformly} in time, while the naive filter does not.  Even if in 
practice the filter is only of interest on a finite time horizon, the 
rapid growth of the error of the naive filter is a severe problem as the 
filter becomes useless after relatively few time steps.  In contrast, 
uniform convergence of the bootstrap filter indicates that its 
approximation error does not accumulate over time, which is essential for 
robust performance.   It is therefore of considerable practical interest 
to establish under what conditions approximate filtering 
algorithms converge uniformly in time.  

The linear example considered here is very special in that the 
filter can be computed exactly.  One would therefore never use a particle 
filter in this setting.  We have chosen an example which admits an exact 
solution as this provides a benchmark with which we can compare the 
performance of particle filter approximations.  On the other hand, exactly 
the same phenomenon as is illustrated in figure \ref{fig:sisr} is observed 
numerically in almost any ergodic filtering problem.  A general 
understanding of this phenomenon is therefore essential in order to 
guarantee reliable performance of approximate filtering algorithms in 
nonlinear filtering problems, which almost never admit an exact 
solution. The aim of this paper is to establish uniform convergence of 
approximate filtering algorithms, and in particular of particle filters,
for a large class of ergodic filtering models.

In the following discussion we denote by $\pi_n$ the conditional 
distribution of $X_n$ given the observation history, and by $\pi_n^N$ its 
particle filter approximation with $N$ particles.  Both are computed 
recursively, which we denote as $\pi_n=F(Y_n,\pi_{n-1}):=F_n\pi_{n-1}$ and 
$\pi_n^N=F^N(Y_n,\pi_{n-1}^N):= F_n^N\pi_{n-1}^N$.

\subsection{Previous work}

Much of what is known about uniform convergence of the particle filter 
has its origins in the work of Del Moral and Guionnet \cite{DmG01}, who 
established a fundamental connection with \emph{filter stability}.  The 
basic idea of this approach is as follows.  The difference between the 
approximate and exact filter can be written as a telescoping sum (setting
for simplicity $\pi_0^N=\pi_0$)
$$
	\pi_n^N - \pi_n = \sum_{k=1}^n\{F_n\cdots F_{k+1}F_k^N\pi_{k-1}^N 
		- F_n\cdots F_{k+1}F_{k}\pi_{k-1}^N\}.
$$
Suppose the the filter is geometrically stable in the following sense:
\begin{equation}
\label{eq:minoriz}
	\|F_n\cdots F_{k+1}\mu-F_n\cdots F_{k+1}\nu\| \le C\,\beta^{n-k}\,
	\|\mu-\nu\|,
\end{equation}
where $\|\,\cdot\,\|$ is a suitable norm on probability measures and 
$C<\infty$, $\beta<1$ are constants.  Then
$$
	\|\pi_n^N - \pi_n\| \le \sum_{k=1}^n C\,\beta^{n-k}\,
	\|F_k^N\pi_{k-1}^N - F_{k}\pi_{k-1}^N\| \le
	\frac{C'}{\sqrt{N}},
$$
where we have used the fact that one time step of the approximate 
filtering algorithm $F_k^N$ introduces an approximation error of order 
$O(N^{-1/2})$ and that the sum over $\beta^{n-k}$ is uniformly bounded.
Thus, evidently, the filter is uniformly convergent at a rate 
$O(N^{-1/2})$.

In order to establish the geometric stability property (\ref{eq:minoriz}) 
of the filter, Del Moral and Guionnet impose the mixing 
assumption $\varepsilon\,\rho(A) \le
\mathbf{P}(X_k\in A|X_{k-1})\le \varepsilon^{-1}\rho(A)$ on the 
signal transition probabilities (for some constant $\varepsilon>0$ and 
probability measure $\rho$) which was originally considered in the filter 
stability context by Atar and Zeitouni \cite{AZ97}.  This is a very strong 
assumption, more stringent even than uniform ergodicity \cite[theorem 
16.0.2]{MT93} of the signal process, and is very difficult to satisfy in 
practice particularly when the signal state space is not compact.  Though 
various methods have been proposed to extend the class of models to which 
the mixing assumption is applicable, essentially all subsequent work on 
uniform convergence of the particle filter 
\cite{LO04,LO03,OR05,Kun05,OR08,OCDM08} has ultimately relied on a form of 
this strong assumption.  Unfortunately, the necessary assumptions are not 
satisfied in many (if not most) models encountered in applications, so 
that the practical applicability of the results established to date 
remains rather limited.

In a sense this conclusion is rather surprising, considering that 
significant progress has been made in recent years in the understanding of 
the filter stability problem (see \cite{CR09} for an extensive review of 
this topic).  For example, Kleptsyna and Veretennikov \cite{KV08} have 
recently established geometric stability $\|F_n\cdots F_{k+1}\mu-F_n\cdots 
F_{k+1}\nu\| \le C(\mu,\nu,Y_{[0,\infty[})\,\beta^{n-k}$ for a particular 
class of non-uniformly ergodic filtering models (see also 
\cite{DFMP08,GLM08} for further variations of this approach), while it has 
been shown that qualitative stability $\|F_n\cdots F_{k+1}\mu-F_n\cdots 
F_{k+1}\nu\|\to 0$ as $n\to\infty$ a.s.\ already holds under minimal 
ergodicity assumptions on the signal \cite{Van08} or under no assumptions 
at all on the signal if the observations are informative \cite{Van08b}. 
The difficulty in applying such results to the uniform convergence problem 
is that the constants in (\ref{eq:minoriz}) are independent of both the 
initial measures $\mu,\nu$ and the observation path $Y_{[0,\infty[}$, 
which is generally not the case when the signal is not uniformly ergodic.
Despite the considerable progress on the filter stability problem, the 
results cited above provide little control over the dependence of the 
constant on the initial measures. This presents a significant hurdle in 
applying these results to the uniform convergence problem.

An entirely different approach for proving uniform convergence properties 
of particle filters was developed by Budhiraja and Kushner \cite{BK01} by 
exploiting certain ergodic properties of nonlinear filters.  Filter 
stability still plays an important role in establishing the ergodic 
theory, but only qualitative stability results are needed, in contrast 
with the quantitative control over the convergence rate and constants 
needed in the approach of Del Moral and Guionnet.  Using the recent 
filter stability results established in \cite{Van08}, the necessary 
ergodic properties can now be established under extremely mild ergodic 
assumptions on the signal process.  In this paper we revisit the approach 
of Budhiraja and Kushner and provide a new set of assumptions for the 
\emph{uniform time average} convergence of bootstrap-type particle filters 
in the following sense ($\|\cdot\|_{\rm BL}$ is the dual 
bounded-Lipschitz norm):
$$
	\lim_{N\to\infty}\,\sup_{T\ge 0}\,\mathbf{E}\left[
	\frac{1}{T}\sum_{k=1}^T\|\pi_k^N-\pi_k\|_{\rm BL}
	\right] = 0.
$$
It should be noted that the time average convergence is weaker than uniform 
convergence established by Del Moral and Guionnet; moreover, this approach 
does not supply a rate of convergence.  On the other hand, we are able to 
demonstrate convergence for a class of non-uniformly ergodic signals
which are presently still out of reach of the more quantitative theory.

\subsection{Organization of the paper}

In section \ref{sec:general} we introduce the basic nonlinear filtering 
problem.  We then develop a general framework for uniform time average 
approximation of the nonlinear filter.
In section \ref{sec:bootstrap} we introduce the bootstrap Monte Carlo 
filtering algorithm and discuss its basic properties.  We show that the 
theory of section \ref{sec:general} can be applied to the bootstrap 
filter, provided that a suitable tightness property can be established.
In section \ref{sec:tightness} we develop two classes of sufficient 
conditions for the requisite tightness property to hold.  Both presume 
that the signal is geometrically ergodic, but different regularity 
assumptions on the observations are required in the two cases to complete 
the proof.  Finally, appendix \ref{app:random} recalls some basic facts 
about weak convergence, while most proofs in the text are postponed to 
appendix \ref{app:proofs}.

\section{A General Approximation Theorem}
\label{sec:general}

The purpose of this section is to introduce the nonlinear filtering 
problem, and to establish a general framework for its approximation 
uniformly in time average (not necessarily by a particle filter). The 
approach of this section follows closely the ideas of Kushner and Huang 
\cite{KH86} and of Budhiraja and Kushner \cite{BK01}, but here we have 
significantly simplified the proofs, generalized the notion of convergence 
and eliminated some technical assumptions.  Our treatment is mostly 
self-contained, but we have postponed the proofs to appendix \ref{app:proofs}.

\subsection{The hidden Markov model and nonlinear filter}

Let $(E,\mathcal{B}(E))$ and $(F,\mathcal{B}(F))$ be Polish spaces endowed 
with their Borel $\sigma$-fields, let 
$P:E\times\mathcal{B}(E)\to[0,1]$ and $\Phi:E\times\mathcal{B}(F)\to[0,1]$ 
be given transition probability kernels, and let 
$\mu:\mathcal{B}(E)\to[0,1]$ be a given probability measure.
We will work with random variables $(X_k,Y_k)_{k\ge 0}$, defined on an 
underlying probability space $(\Omega,\mathcal{F},\mathbf{P})$,
such that $(X_n)_{n\ge 0}$ is a Markov chain with initial 
measure $X_0\sim\mu$ and transition probability $P$, and such that
$(Y_n)_{n\ge 0}$ are conditionally independent given $(X_n)_{n\ge 0}$
with $\mathbf{P}(Y_n\in A|X_n) = \Phi(X_n,A)$. Such a model can always be 
constructed in a canonical fashion, and is called a \textit{hidden Markov 
model} with initial measure $\mu$, transition kernel $P$ and observation 
kernel $\Phi$.  

We will make the following nondegeneracy assumption on the observation 
kernel.

\begin{aspt}[Nondegeneracy]
\label{aspt:nondeg}
There is a $\sigma$-finite measure $\varphi:\mathcal{B}(F)\to\mathbb{R}$ 
and a strictly positive measurable function $\Upsilon:E\times 
F\to\mbox{}]0,\infty[\mbox{}$ such that
$$
        \Phi(x,A) = \int_A \Upsilon(x,y)\,\varphi(dy)
        \quad\mbox{for all }x\in E,~A\in\mathcal{B}(F).
$$
\end{aspt}

We now define the probability kernels $\pi_{k-}:
F^{k}\times\mathcal{B}(E)\to[0,1]$ and
$\pi_k:F^{k+1}\times\mathcal{B}(E)\to[0,1]$ by the following recursion:
for all $y_0,\ldots,y_k\in F$ and $A\in\mathcal{B}(E)$, we have
$$
        \pi_{k-}(y_{0\ldots k-1},A) = 
        \int P(x,A)\,\pi_{k-1}(y_{0 \ldots k-1},dx),\qquad
        \pi_k(y_{0 \ldots k},A) = \frac{\int_A \Upsilon(x,y_k)\,
                \pi_{k-}(y_{0 \ldots k-1},dx)}{
        \int \Upsilon(x,y_k)\,\pi_{k-}(y_{0 \ldots k-1},dx)}
$$
with the initial condition $\pi_{0-}(A) = \mu(A)$.  Then it is well known 
that by the Bayes formula,
$$
        \mathbf{P}(X_k\in A|Y_0,\ldots,Y_{k-1}) = \pi_{k-}(Y_{0 \ldots
                k-1},A),\qquad\quad
        \mathbf{P}(X_k\in A|Y_0,\ldots,Y_k) =
        \pi_k(Y_{0 \ldots k},A).
$$
For notational convenience we will simply write
$\pi_{k-}(A)=\pi_{k-}(Y_{0 \ldots k-1},A)$ and
$\pi_k(A)=\pi_k(Y_{0 \ldots k},A)$.
The kernel $\pi_k$ is called the \textit{nonlinear filter} and $\pi_{k-}$ 
is the \textit{one step predictor} associated with the hidden Markov model
$(X_k,Y_k)_{k\ge 0}$.  Unfortunately, these infinite dimensional 
quantities are typically not explicitly computable.  We aim to obtain
a computationally tractable approximation.

\subsection{Markov and ergodic properties}

In the following, we denote by $\mathcal{P}(E)$ the space of probability 
measures on $(E,\mathcal{B}(E))$ endowed with the topology of weak 
convergence of probability measures and the associated Borel 
$\sigma$-field.  We define on $\mathcal{P}(E)$ the probability distances
$$
	\|\nu-\nu'\|_{\rm BL} = \sup_{f\in\mathrm{Lip}(E)}
	\left|\int f\,d\nu-\int f\,d\nu'\right|,
	\qquad
	\|\nu-\nu'\|_{\rm TV} = \sup_{\|f\|_\infty\le 1}
	\left|\int f\,d\nu-\int f\,d\nu'\right|,
$$
where we have defined $\mathrm{Lip}(E)=\{f:\|f\|_\infty\le 1,~\|f\|_L\le 
1\}$ and $\|f\|_L$ is the Lipschitz constant of $f$.  The dual 
bounded-Lipschitz distance $\|\cdot\|_{\rm BL}$ metrizes the weak 
convergence topology on $\mathcal{P}(E)$, while the total variation 
distance $\|\cdot\|_{\rm TV}$ is strictly stronger.

Let us recall that any probability kernel $\nu:\Omega\times\mathcal{B}(E)
\to[0,1]$ can equivalently be viewed as a $\mathcal{P}(E)$-valued random 
variable on the measure space $\Omega$ (see, e.g., \cite[lemma 1.40]{Kal02}).
In particular, we may consider the filter $(\pi_{k})_{k\ge 0}$ to be a 
$\mathcal{P}(E)$-valued stochastic process adapted to the filtration 
$\mathcal{F}_k^Y= \sigma\{Y_0,\ldots,Y_k\}$.  It is well known that this 
process possesses the Markov property, see, e.g., \cite{Ste89}, and the 
associated ergodic theory will play a key role in the following.

\begin{aspt}[Ergodicity]
\label{aspt:harris}
$(X_k)_{k\ge 0}$ is positive Harris recurrent and aperiodic, i.e., there 
is a (unique) $P$-invariant measure $\lambda\in\mathcal{P}(E)$ such that 
$\|\nu P^k-\lambda\|_{\rm TV}\to 0$ as $k\to\infty$ for every 
$\nu\in\mathcal{P}(E)$.
\end{aspt}

When assumption \ref{aspt:nondeg} holds, we may define the update map 
$\mathsf{U}:F\times\mathcal{P}(E)\to\mathcal{P}(E)$ as
$$
	\mathsf{U}(y,\pi)(A) = 
	\frac{\int I_A(x)\,\Upsilon(x,y)\,\pi(dx)}
	{\int \Upsilon(x,y)\,\pi(dx)}.
$$
The following result collects the various properties of the filter 
that will be used below.  

\begin{prop}
\label{prop:mfe}
Suppose that assumption \ref{aspt:nondeg} holds.  Then the 
$E\times\mathcal{P}(E)$-valued stochastic process 
$(X_k,\pi_{k})_{k\ge 0}$ is Markov with transition kernel $\mathsf{\Pi}:
E\times\mathcal{P}(E)\times\mathcal{B}(E\times\mathcal{P}(E))\to[0,1]$,
$$
	\int f(x',\pi')\,\mathsf{\Pi}(x,\pi,dx',d\pi') =
	\int f(x',\mathsf{U}(y,\pi P))\,\Upsilon(x',y)\,\varphi(dy)\,
	P(x,dx'),
$$
and initial measure $M\in\mathcal{P}(E\times\mathcal{P}(E))$,
$$
	\int f(x,\pi)\,M(dx,d\pi) = 
	\int f(x,\mathsf{U}(y,\mu))\,\Upsilon(x,y)\,\varphi(dy)\,
	\mu(dx).
$$
Moreover, if assumption \ref{aspt:harris} holds, 
then $\mathsf{\Pi}$ possesses a unique invariant measure 
$\Lambda\in\mathcal{P}(E\times\mathcal{P}(E))$.
\end{prop}

The proof is given in appendix \ref{app:prop:mfe}.  Let us remark that the 
Markov property is elementary, while uniqueness of the invariant measure 
hinges on recent progress on the filter stability problem \cite{Van08}.

\subsection{A general approximation theorem}

As the filter $\pi_k$ can not be computed exactly in practice, we aim to 
approximate it by a sequence of computationally tractable approximate 
filters $\pi_k^N$ ($N\in\mathbb{N}$), such that $\pi_k^N\to \pi_k$ as 
$N\to\infty$.  The goal of this section is to investigate what assumptions 
should be imposed on the filter approximations so that they converge to 
the exact filter uniformly in time average.  We will subsequently apply 
this result to the setting where $\pi_k^N$ is a bootstrap type Monte Carlo 
particle filter with $N$ particles. However, the results of this section 
are general and could be applied to other types of filter approximation 
also.

We have seen in the previous section that $(\pi_k)_{k\ge 0}$ is a 
$\mathcal{P}(E)$-valued $\mathcal{F}_k^Y$-adapted process, such that 
$(X_k,\pi_k)_{k\ge 0}$ is Markov.  We will consider approximate filters 
$\pi_k^N$ of a similar type, but we allow them to be adapted to a slighly 
larger filtration.  This is needed to account for the random sampling step 
in Monte Carlo particle filters, which introduces additional randomness
into the algorithm.

\begin{aspt}[Approximation]
\label{aspt:approx}
For every $N\in\mathbb{N}$, the process $(\pi_k^N)_{k\ge 0}$
satisfies the following.
\begin{enumerate}
\item $(\pi_k^N)_{k\ge 0}$ is a $\mathcal{P}(E)$-valued 
$\mathcal{F}_k^Y\vee\mathcal{G}$-adapted process, where
$\mathcal{G}$ is independent of $(X_k,Y_k)_{k\ge 0}$.
\item $(X_k,\pi_k^N)_{k\ge 0}$ is Markov with transition kernel
$\mathsf{\Pi}_N$ and initial measure $M_N$.
\end{enumerate}
\end{aspt}

We obtain the following general approximation theorem.

\begin{thm}
\label{thm:approx}
Suppose that assumptions \ref{aspt:nondeg}--\ref{aspt:approx} hold.
Moreover, we make the following one step convergence and tightness
assumptions on the approximating sequence.
\begin{enumerate}
\item For any bounded continuous 
$F:E\times\mathcal{P}(E)\to\mathbb{R}$ and $x_N\to x$, 
$\nu_N\Rightarrow\nu$ as $N\to\infty$, we have 
$$
        \int F(x',\nu')\,\mathsf{\Pi}_{N}(x_N,\nu_N,dx',d\nu')
        \xrightarrow{N\to\infty}
        \int F(x',\nu')\,\mathsf{\Pi}(x,\nu,dx',d\nu').
$$
In addition, we have $M_N\Rightarrow M$ as $N\to\infty$.
\item For any sequence $T_N\nearrow\infty$ as $N\to\infty$,
$$
        \mbox{the family of probability measures}\quad
        \Xi_N(A) = \mathbf{E}\left[\frac{1}{T_N}
        \sum_{k=1}^{T_N}\pi_{k}^{N}(A)\right],\quad
        N\ge 1\quad
        \mbox{is tight}.
$$
\end{enumerate}
Then the sequence $(\pi_k^N)_{k\ge 0}$ converges to
$(\pi_k)_{k\ge 0}$ as $N\to\infty$ uniformly in time average:
$$
        \lim_{N\to\infty}\sup_{T\ge 0}
        \mathbf{E}\left[
        \frac{1}{T}\sum_{k=1}^{T}\|\pi^{N}_{k}-\pi_{k}\|_{\rm BL}
        \right]=0.
$$
\end{thm}

The proof of this theorem is given in appendix \ref{app:thm:approx}.  

Let us note that the uniform time average convergence guaranteed by the 
theorem allows us to answer related convergence questions as well. For 
example, we can prove that the time average mean square error of the 
estimates obtained from the approximate filter converges to the time 
average mean square error of the estimates obtained from the exact filter, 
uniformly in time.

\begin{cor}
\label{cor:esterr}
Suppose that the assumptions of theorem \ref{thm:approx} are satisfied.  
Then
$$
	\lim_{N\to\infty}\,
	\sup_{T\ge 0}\,
	\mathbf{E}\left[\left|
        \frac{1}{T}\sum_{k=1}^{T}
	\left(f(X_k)-\int f\,d\pi_k^N\right)^2 -
        \frac{1}{T}\sum_{k=1}^{T}
	\left(f(X_k) -
	\int f\,d\pi_k\right)^2\right|
	\right] = 0
$$
for any bounded continuous function $f$.
\end{cor}

The proof is given in appendix \ref{app:cor:esterr}.

\begin{rem} \textit{The one step convergence assumption.}
The first condition of theorem \ref{thm:approx} ensures that the 
approximate filter converges to the exact filter on any finite time 
horizon (lemma \ref{lem:finhorz}).  This is certainly a minimal 
requirement for convergence, and is typically easily verified in practice.  
\end{rem}

\begin{rem} \textit{The tightness assumption.}
The second condition of theorem \ref{thm:approx} ensures, roughly 
speaking, that the approximate filter does not lose mass to infinity after 
a long time (at least on average with respect to time and the 
observations).  This is certainly the case for the signal itself by 
assumption \ref{aspt:harris}, and this property is inherited by the exact 
filter by virtue of lemma \ref{lem:bary}.  Tightness of the approximate 
filter is not automatic, however, and needs to be imposed separately.  
Though this, too, is arguably a minimal assumption to ensure convergence 
of the approximate filters, the tightness property appears to be much more 
difficult to demonstrate in practice.  Indeed, this is the main difficulty 
in applying theorem \ref{thm:approx} to Monte Carlo particle filters.

An exception is the case where the signal state space $E$ is compact; we 
state this as a lemma, though the result is entirely obvious and requires 
no proof.

\begin{lem}
\label{lem:cpct}
If $E$ is compact, then the second condition of theorem \ref{thm:approx} 
is automatically satisfied.
\end{lem}

In the compact setting, however, the generality of the ergodic assumption 
\ref{aspt:harris} is slightly misleading.  Indeed, note that the 
first condition of theorem \ref{thm:approx} implies that the signal 
transition kernel $P$ is Feller.  Therefore, under the mild assumption 
that the support of the signal invariant measure $\lambda$ has nonempty 
interior, compactness of the state space implies that the signal is even
\emph{uniformly ergodic} \cite[theorem 16.2.5 and theorem 6.2.9]{MT93}.
Moreover, if we assume that $x\mapsto\Upsilon(x,y)$ is continuous for 
every $y$ (as we will do in order to prove the first condition of theorem 
\ref{thm:approx}), assumption \ref{aspt:nondeg} and compactness of $E$ 
implies that $\Upsilon(\cdot,y)$ is bounded away from zero for every $y$.
In this setting, uniform convergence could be studied more directly using 
the techniques in \cite{DmG01}.

When $E$ is not compact, a sufficient condition for tightness is the 
following.

\begin{lem}
\label{lem:tight}
If the family $\{\mathbf{E}\pi_k^N:k,N\ge 1\}$ is tight,
the second condition of theorem \ref{thm:approx} holds.
\end{lem}

We omit the proof, which is straightforward.
\end{rem}

\section{The Bootstrap Particle Filter}
\label{sec:bootstrap}

The practical problem in implementing the exact filter is that the 
conditional distribution $\pi_k$ is an infinite dimensional object.
In applying the theory, one must therefore seek finite dimensional 
approximations.  The idea behind particle filters is to approximate
the nonlinear filter by atomic measures with a fixed 
number of particles $N\in\mathbb{N}$, i.e., by measures in the space
$$
        \mathcal{P}_N(E)=\left\{
        \sum_{i=1}^Nw_i\delta_{x_i}:
        x_1,\ldots,x_N\in E,~
        w_1,\ldots,w_N\ge 0,~
        \sum_{i=1}^Nw_i=1\right\}\subset\mathcal{P}(E).
$$
Note that the filtering recursion does not naturally leave the set 
$\mathcal{P}_N(E)$ invariant; therefore, approximation is unavoidable.
The bootstrap particle filter introduces an additional \textit{sampling} 
step in the filtering recursion to project the filter back into the set 
$\mathcal{P}_N(E)$.

To be precise, define the \emph{sampling transition 
kernel} $\mathsf{R}_N:\mathcal{P}(E)\times
\mathcal{B}(\mathcal{P}(E))\to[0,1]$ as
$$
        \int F(\nu)\,\mathsf{R}_N(\rho,d\nu) =
        \int F\left(\frac{1}{N}\sum_{i=1}^N\delta_{x_i}\right)
        \rho(dx_1)\cdots\rho(dx_N).
$$
Then $\mathsf{R}_N(\rho,\cdot)$ is the law of a $\mathcal{P}(E)$-valued 
random variable $\varrho$ that is generated as follows:
\begin{enumerate}
\item Sample $N$ i.i.d.\ random variables $X^1,\ldots,X^N$ from $\rho$.
\item Set $\varrho = \frac{1}{N}\{\delta_{X^1}+\cdots+\delta_{X^N}\}$.
\end{enumerate}
We now introduce the transition kernel for the 
bootstrap particle filter as
$$
	\int f(x',\pi')\,\mathsf{\Pi}_N(x,\pi,dx',d\pi') =
	\int f(x',\mathsf{U}(y,\pi'))\,\mathsf{R}_N(\pi P,d\pi')\,
	\Upsilon(x',y)\,\varphi(dy)\,P(x,dx'),
$$
and we define the initial measure for the bootstrap particle filter as
$$
	\int f(x,\pi)\,M_N(dx,d\pi) = 
	\int f(x,\mathsf{U}(y,\pi))\,\mathsf{R}_N(\mu,d\pi)\,
	\Upsilon(x,y)\,\varphi(dy)\,\mu(dx).
$$
Note, in particular, that by construction $M_N$ and 
$\mathsf{\Pi}_N(x,\pi,\cdot)$ are supported on $E\times\mathcal{P}_N(E)$ 
for any $x,\pi$, so that the bootstrap particle filter is indeed finite 
dimensional in nature.  Moreover, the law of large numbers strongly 
suggests convergence to the exact filter as $N\to\infty$ at least on 
finite time intervals; we will make this precise below by verifying the 
first condition of theorem \ref{thm:approx}.

We have not yet introduced an explicit construction of the random 
variables $(\pi_k^N)_{k\ge 0}$ on the probability space 
$(\Omega,\mathcal{F},\mathbf{P})$.  However, as all our state spaces are 
Polish, it is a standard fact (e.g., along the lines of \cite[proposition 
8.6]{Kal02}) that the joint process $(X_k,Y_k,\pi_k,\pi_k^N)_{k\ge 0}$ can 
be obtained for any $N\ge 1$ by a canonical construction, provided the 
probability space $(\Omega,\mathcal{F},\mathbf{P})$ carries a countable 
family of i.i.d.\ $\mathrm{Unif}(0,1)$-random variables $(\zeta_k)_{k\ge 0}$ 
independent of $(X_k,Y_k)_{k\ge 0}$.  The random variables 
$(\zeta_k)_{k\ge 0}$ provide the additional randomness introduced by 
the sampling steps in the bootstrap filtering algorithm, and the 
construction is such that $\pi_k^N$ is $\mathcal{F}_k^Y\vee\mathcal{G}$-adapted 
with $\mathcal{G}=\sigma\{\zeta_k:k\ge 0\}$.  As it will not be 
needed in what follows, the construction of $(X_k,Y_k,\pi_k,\pi_k^N)_{k\ge 0}$
will be left implicit, but the details of the construction should be 
evident from the bootstrap filtering algorithm \ref{alg:boot} (which is 
clearly very straightforward to implement in practice).

\begin{algorithm}[t]
\SetLine
Sample i.i.d.\ $x_0^i$, $i=1,\ldots,N$ from the initial distribution $\mu$\;
Compute $w_0^i=\Upsilon(x_0^i,Y_0)/
	\sum_{\ell=1}^N\Upsilon(x_0^\ell,Y_0)$, $i=1,\ldots,N$\;
Set $\pi_0^N = \sum_{i=1}^Nw_0^i\delta_{x_0^i}$\;
\For{k=1,\ldots,n}{
Sample i.i.d.\ $\tilde x_{k-1}^i$, $i=1,\ldots,N$ from the distribution
$\pi_{k-1}^N$\;
Sample $x_k^i$ from $P(\tilde x_{k-1}^i,\,\cdot\,)$, $i=1,\ldots,N$\;
Compute $w_k^i=\Upsilon(x_k^i,Y_k)/\sum_{\ell=1}^N
	\Upsilon(x_k^\ell,Y_k)$, $i=1,\ldots,N$\;
Set $\pi_k^N = \sum_{i=1}^Nw_k^i\delta_{x_k^i}$\;
}
\caption{Bootstrap Filtering Algorithm}
\label{alg:boot}
\end{algorithm}

\begin{rem}
A conceptually simpler \emph{naive} particle filter could be constructed 
as follows.  By the Bayes formula, the exact filter at time $k$ can be 
expressed as
$$
	\pi_k(y_0,\ldots,y_k,A) =
	\frac{\mathbf{E}(I_A(X_k)\Upsilon(X_k,y_k)\cdots\Upsilon(X_0,y_0))}{
	\mathbf{E}(\Upsilon(X_k,y_k)\cdots\Upsilon(X_0,y_0))}.
$$
Therefore, by the law of large numbers, we can approximate $\pi_k$ as
follows:
$$
	\pi_k(y_0,\ldots,y_k,A) \approx
	\frac{\sum_{i=1}^N
	I_A(X_k^i)\Upsilon(X_k^i,y_k)\cdots\Upsilon(X_0^i,y_0)}{
	\sum_{i=1}^N\Upsilon(X_k^i,y_k)\cdots\Upsilon(X_0^i,y_0)},
$$
where $(X_0^i,\ldots,X_k^i)$, $i=1,\ldots,N$ are i.i.d.\ samples from the 
law of $(X_0,\ldots,X_k)$.  Indeed, by the law of large numbers, this 
approximation is immediately seen to converge to the exact filter as 
$N\to\infty$.  However, as can be seen in the numerical example in figure 
\ref{fig:sisr}, the convergence is not uniform in time, and in fact the 
performance is quite poor (see \cite{DJ01} for a theoretical perspective).
\end{rem}

Our aim is to prove that the bootstrap particle filter converges uniformly 
in time average.  We will do this by verifying the conditions of theorem 
\ref{thm:approx}.  Clearly assumption \ref{aspt:approx} holds by 
construction, while assumptions \ref{aspt:nondeg} and \ref{aspt:harris} on 
the filtering model will be presumed from the outset.  We now show that 
the first condition of theorem \ref{thm:approx} holds under a mild 
continuity assumption on the filtering model.  Tightness is a much more 
difficult problem, and will be tackled in the next section.

\begin{aspt}[Continuity]
\label{aspt:feller}
The following hold:
\begin{enumerate}
\item $P$ is Feller, i.e., $x\mapsto P(x,\,\cdot\,)$ is continuous;
\item For every $y\in F$, the map $x\mapsto\Upsilon(x,y)$ is 
continuous and bounded.
\end{enumerate}
\end{aspt}

\begin{prop}
\label{prop:bootconv}
Suppose that assumptions \ref{aspt:nondeg} and \ref{aspt:feller} hold.
Then the first condition of theorem \ref{thm:approx} holds true for the 
bootstrap particle filter.  In particular, 
$\mathbf{E}(\|\pi_{k}^{N}-\pi_{k}\|_{\rm BL})\xrightarrow[N\to\infty]{}0$
for any $k<\infty$.
\end{prop}

The proof of this result is given in appendix \ref{app:prop:bootconv}.
From theorem \ref{thm:approx}, we immediately obtain:

\begin{cor}
\label{cor:bootconv}
Suppose that assumptions \ref{aspt:nondeg}, \ref{aspt:harris}, and
\ref{aspt:feller} hold, and that
$$
        \mbox{the family of probability measures}\quad
        \Xi_N(A) = \mathbf{E}\left[\frac{1}{T_N}
        \sum_{k=1}^{T_N}\pi_{k}^{N}(A)\right],\quad
        N\ge 1\quad
        \mbox{is tight}
$$
for any sequence $T_N\nearrow\infty$ as $N\to\infty$.  Then
$$
        \lim_{N\to\infty}\sup_{T\ge 0}
        \mathbf{E}\left[
        \frac{1}{T}\sum_{k=1}^{T}\|\pi^{N}_{k}-\pi_{k}\|_{\rm BL}
        \right]=0
$$
holds true for the bootstrap particle filter.
\end{cor}

\section{Sufficient Conditions for Tightness}
\label{sec:tightness}

By corollary \ref{cor:bootconv}, all that remains to prove in order to 
establish uniform time average consistency of the boostrap particle filter 
is the tightness of particle system generated by the algorithm---i.e., we 
must rule out the possibility that the particle system loses mass to 
infinity after running for a long time.  It seems intuitively plausible 
that this can be proved under rather general conditions, as both the 
signal and filter are already ergodic (see assumption \ref{aspt:harris} 
and \cite{Van08}) and the sampling step in the bootstrap algorithm does 
not change the center of mass of the filter.

Unfortunately, the tightness problem appears to be much more difficult 
than one might expect.  A rather ominous counterexample in a different 
setting \cite{RRS98} shows that, contrary to intuition, arbitrarily small 
perturbations may cause a Markov chain to become transient (and hence lose 
its tightness property) even when the unperturbed chain is geometrically 
ergodic.  Though the implications to the present setting are unclear, such 
examples suggest that the problem may be delicate and that tightness can 
not be taken for granted.  In this section, we will provide two sets of 
general sufficient conditions under which tightness can be verified for 
the bootstrap particle filter.  Both sets of conditions require geometric 
ergodicity of the signal (which is stronger than assumption 
\ref{aspt:harris}), and each imposes a different set of restrictions on 
the observation structure.

\begin{rem} 
Assumptions \ref{aspt:nondeg}, \ref{aspt:harris}, and \ref{aspt:feller} 
are very mild and are satisfied by the majority of ergodic filtering 
problems.  In contrast, the sufficient conditions for tightness 
below are rather restrictive, and in this sense our results are not 
entirely satisfactory---establishing tightness under minimal ergodicity 
and observation assumptions remains an open problem.  Nonetheless, the 
tightness property is purely qualitative and thus appears to be 
significantly more tractable than the quantitative controls required in 
other approaches to the uniform convergence problem (indeed, the general
conditions imposed below are still out of reach of other approaches).  
Another interesting possibility is that tightness might be achieved by 
introducing suitable modifications to the bootstrap filtering algorithm, 
e.g., by means of a periodic resampling scheme or using some form of 
regularization.
\end{rem}

Let us briefly recall the relevant notion of geometric ergodicity.
A function $V:E\to[1,\infty[\mbox{}$ is said to possess \emph{compact 
level sets} if the set $\{x\in E:V(x)\le r\}$ is compact for every 
$r\ge 1$.  Given such a function $V$, we define the $V$-total variation 
distance between $\mu,\nu\in\mathcal{P}(E)$ as
$$
	\|\mu-\nu\|_V =
	\sup_{|f|\le V}\left|\int f\,d\mu-\int f\,d\nu\right|
	= \int V\,d|\mu-\nu|.
$$
We will call the Markov chain $(X_k)_{k\ge 0}$ \emph{geometrically 
ergodic} if there is a function $V:E\to[1,\infty[\mbox{}$ with compact 
level sets, a $P$-invariant measure $\lambda$, and constants $C<\infty$ 
and $\beta<1$ such that
$$
	\|P^k(x,\cdot)-\lambda\|_V \le C\,V(x)\,\beta^k
	\quad\mbox{for all }x\in E.
$$
Note that geometric ergodicity is strictly stronger than assumption 
\ref{aspt:harris}.  Geometric ergodicity is often easily verified in terms 
of Lyapunov-type conditions on the transition kernel and is satisfied in 
many practical applications; see the monograph \cite{MT93} for an 
extensive development of this theory.

\subsection{Case I: bounded observations}

We will first consider the following assumptions.

\begin{aspt}[Tightness: Case I]
\label{aspt:casei}
The following hold.
\begin{enumerate}
\item The signal is geometrically ergodic ($\|P^k(x,\cdot)-\lambda\|_V 
\le C\,V(x)\,\beta^k$, $V$ has compact level sets).
\item There exist strictly positive functions 
$u_+,u_-:F\to\mbox{}]0,\infty[\mbox{}$ such that
$$
	u_-(y)\le\Upsilon(x,y)\le u_+(y)\quad
	\mbox{for all }x\in E,\qquad\qquad
	\int \frac{u_+(y)^2}{u_-(y)}\,\varphi(dy) < \infty.
$$
\end{enumerate}
\end{aspt}

Assumption \ref{aspt:casei} is typically satisfied when the observations 
are of the additive noise type with a bounded observation function.  As an 
example, consider the observation model $Y_k=h(X_k)+\xi_k$ on the observation 
state space $F=\mathbb{R}^d$, where $\xi_k$ are 
i.i.d.\ $N(0,\Sigma)$-random variables independent of $(X_k)_{k\ge 0}$
for some strictly positive covariance matrix $\Sigma$, and 
$h:E\to\mathbb{R}^d$ is a continuous and bounded observation function.  
Then we can set
$$
	\varphi(dy) = \frac{1}{(2\pi)^{d/2}\,|\Sigma|^{1/2}}
	\,\exp\left(-\frac{1}{2}\,y^*\Sigma^{-1}y\right)dy,
	\qquad
	\Upsilon(x,y) = \exp\left(
		y^*\Sigma^{-1}h(x)-\frac{1}{2}\,h(x)^*\Sigma^{-1}h(x)
	\right),
$$
and assumptions \ref{aspt:nondeg} and \ref{aspt:feller} are clearly 
satisfied for this observation model.  Moreover, evidently
$$
	u_-(y) = \exp\left(
        -\left[\|y\|+\frac{1}{2}\,\|h\|_\infty\right]
	\|\Sigma^{-1}\|\,\|h\|_\infty
        \right),\qquad
	u_+(y) = \exp\left(
	\|y\|\,\|\Sigma^{-1}\|\,\|h\|_\infty\right),
$$
where $\|h\|_\infty = \sup_{x\in E}\|h(x)\|$, satisfy the requirement
in assumption \ref{aspt:casei}.

\subsection{Case II: strongly unbounded observations}

To satisfy assumption \ref{aspt:casei}, the observation function $h$ will 
generally need to be bounded.  Our second set of assumptions is 
essentially the opposite scenario: we consider an observation model where 
$h$ is strongly unbounded, i.e., converges to infinity in every direction 
(the requirement below that $\|h\|$ has compact level sets).

\begin{aspt}[Tightness: Case II]
\label{aspt:caseii}
Let $F=\mathbb{R}^d$, and suppose that $Y_k=h(X_k)+\sigma(X_k)\,\xi_k$
where $\xi_k$ are i.i.d.\ random variables independent of $(X_k)_{k\ge 0}$.
We assume the following:
\begin{enumerate}
\item The signal is geometrically ergodic ($\|P^k(x,\cdot)-\lambda\|_V 
\le C\,V(x)\,\beta^k$, $V$ has compact level sets).
\item $h:E\to\mathbb{R}^d$, $\sigma:E\to\mathbb{R}^{d\times d}$ are
continuous, $\varepsilon\|v\|\le\|\sigma(x)v\|\le\varepsilon^{-1}\|v\|$
$\forall\,x,v$ for some $\varepsilon>0$.
\item The law of the observation noise $\xi_k$ has a strictly positive, 
bounded and continuous density $q_\xi:\mathbb{R}^d\to\mbox{}]0,\infty[\mbox{}$
with respect to the Lebesgue measure on $\mathbb{R}^d$.
\item There is a nonincreasing 
$q:[0,\infty[\mbox{}\to\mbox{}]0,\infty[\mbox{}$, a norm $|\cdot|$ on 
$\mathbb{R}^d$, and $a_1,a_2>0$ such that
$$
	a_1\,q(|z|)\le q_\xi(z)\le a_2\,q(|z|)
	\quad\mbox{for all }z\in\mathbb{R}^d.
$$
\item There are constants $b_1,b_3>0$, $b_2,b_4\in\mathbb{R}$, and $p>0$
with $\mathbf{E}(\|\xi_k\|^p)<\infty$, such that
$$
	b_1\|h(x)\|^p+b_2\le V(x)\le b_3\|h(x)\|^p+b_4
	\quad\mbox{for all }x\in E.
$$
\end{enumerate}
\end{aspt}

\begin{rem}
Note that when assumption \ref{aspt:caseii} is satisfied, we may always 
choose $\varphi$ to be the Lebesgue measure and $\Upsilon(x,y) = 
q_\xi(\sigma(x)^{-1}\{y-h(x)\})$, which is strictly positive and 
$x\mapsto\Upsilon(x,y)$ is bounded and continuous for every $y$.  We 
therefore automatically satisfy assumption \ref{aspt:nondeg} and the 
observation part of assumption \ref{aspt:feller}.  Moreover, geometric 
ergodicity implies that assumption \ref{aspt:harris} holds also.
Finally, note that as $V$ is by definition presumed to have compact level 
sets, the assumption implies that $x\mapsto\|h(x)\|$ has compact level 
sets also, i.e., $h(x)$ is strongly unbounded.
\end{rem}

A typical example where assumption \ref{aspt:caseii} is satisfied is the 
following.  Let $E=F=\mathbb{R}^d$, and consider the observation model 
$Y_k=h(X_k)+\xi_k$ where $\xi_k\sim N(0,\Sigma)$ for some strictly 
positive covariance matrix $\Sigma$, and $h(x)=h_0(x)+h_1(x)$ where
$h_0$ is bi-Lipschitz (i.e., it is Lipschitz, invertible, and its inverse 
is Lipschitz) and $h_1$ is a bounded continuous function.  Moreover, 
assume that the signal is geometrically ergodic where $V$ satisfies the 
growth condition $b_1'\|x\|^p+b_2'\le V(x)\le b_3'\|x\|^p+b_4'$ for some 
$p,b_1',b_3'>0$.  Let us verify the requirements of assumption 
\ref{aspt:caseii} in this setting.

First, the law of $\xi_k$ has a density $q_\xi(z)=
\exp(-z^*\Sigma^{-1}z/2)/(2\pi)^{d/2}|\Sigma|^{1/2}$ with respect to the 
Lebesgue measure.  Therefore $q_\xi$ is bounded, continuous, and strictly 
positive, and we may evidently set $|z|^2 = z^*\Sigma^{-1}z$ (which 
defines a norm), $q(v)=\exp(-v^2/2)$ (which is nonincreasing), and 
$a_1=a_2=(2\pi)^{-d/2}|\Sigma|^{-1/2}$.  Moreover, it is easily 
established that
$$
	l_1\|x\| - \|h_0(0)\| - \|h_1\|_\infty \le
	\|h(x)\| \le
	l_2\|x\| + \|h_0(0)\| + \|h_1\|_\infty,
$$
where we have used that $l_1\|x-z\|\le \|h_0(x)-h_0(z)\|\le l_2\|x-z\|$ 
for some $l_1,l_2>0$ by the bi-Lipschitz property of $h_0$.  We may 
therefore estimate
$$
	\frac{b_1'}{C_pl_2^p}\|h(x)\|^p 
	- \frac{b_1'\alpha^p}{l_2^p}
	+ b_2'
	\le V(x)\le
	\frac{C_p b_3'}{l_1^{p}}\|h(x)\|^p+\frac{C_p b_3'
		\alpha^p}{l_1^p}+b_4',
$$
where we have written $(a+b)^p\le C_p(a^p+b^p)$ for $a,b\ge 0$ (one can 
choose $C_p=\max(1,2^{p-1})$) and $\alpha=\|h_0(0)\| + \|h_1\|_\infty$.
Finally, as any Gaussian has finite moments, 
$\mathbf{E}(\|\xi_k\|^p)<\infty$.  

\subsection{Uniform time average consistency}

We have now introduced two sets of assumptions on the filtering model.
Our main result states that either of these assumptions is sufficient for 
uniform time average consistency of the bootstrap particle filter.

\begin{thm} 
\label{thm:main} 
Suppose that either assumptions \ref{aspt:nondeg}, \ref{aspt:feller} and 
\ref{aspt:casei} hold, or that the signal transition kernel $P$ is Feller 
and that assumption \ref{aspt:caseii} holds.  In addition, suppose that 
$\mu(V)<\infty$.  Then the tightness assumption of corollary 
\ref{cor:bootconv} holds, and in particular
$$
        \lim_{N\to\infty}\sup_{T\ge 0}
        \mathbf{E}\left[
        \frac{1}{T}\sum_{k=1}^{T}\|\pi^{N}_{k}-\pi_{k}\|_{\rm BL}
        \right]=0
$$
holds true for the bootstrap particle filter.
\end{thm}

The proof is given in appendix \ref{app:thm:main}.

\appendix

\section{Some Basic Facts on Weak Convergence}
\label{app:random}

The purpose of this appendix is to recall some basic facts on weak 
convergence of probability measures and transition kernels that are 
particularly useful in the setting of this paper.

\subsection{Weak convergence of kernels}

We begin by showing that weak convergence of transition probability 
kernels, in a sufficiently strong sense, can be iterated.

\begin{lem}
\label{lem:iter}
Let $K_N:E\times\mathcal{B}(E)\to[0,1]$, $N\in\mathbb{N}$ be a sequence of 
transition kernels on a Polish space $E$, and let $K$ be another such 
kernel.  Then for every bounded continuous $f:E\to\mathbb{R}$
$$
	\int f(z)\,K_{N}(x_N,dz)\xrightarrow{N\to\infty}
	\int f(z)\,K(x,dz)\qquad
	\mbox{whenever }x_N\xrightarrow{N\to\infty}x
$$
if and only if for any $j\ge 1$, we have $\nu_NK_N^j\Rightarrow \nu K^j$ 
as $N\to\infty$ whenever $\nu_N\Rightarrow\nu$.
\end{lem}

\begin{proof}
The if part follows trivially by choosing $\nu_N=\delta_{x_N}$,
$\nu=\delta_x$, and $j=1$.  To prove the only if part, suppose we have 
established that the result holds for $j\le k$.  Then it clearly holds 
also for $j\le k+1$.  By induction, it therefore suffices to consider the 
case $j=1$.

As $\nu_N\Rightarrow\nu$, we can construct using the Skorokhod 
representation theorem a sequence of random variables
$X^N\to X$ a.s.\ such that $X^N\sim\nu_n$, $X\sim\nu$.
Let $f$ be bounded and continuous, and note that
$\nu_NK_Nf=\mathbf{E}(K_Nf(X^N))$ and
$\nu Kf=\mathbf{E}(Kf(X))$.  But by our assumption $K_Nf(X^N)\to Kf(X)$
a.s., so the claim follows immediately using dominated convergence.
\end{proof}

\subsection{Tightness of random measures}

As many of the stochastic processes in this paper are measure-valued,
we require a simple condition for tightness of a family of measure-valued 
random variables.  The following necessary and sufficient condition is 
quoted from \cite[corollary 2.2]{Jak88}.  As usual, if $\varrho$ is a 
$\mathcal{P}(E)$-valued random variable, we denote by 
$\rho=\mathbf{E}\varrho\in\mathcal{P}(E)$ the probability measure defined 
by $\rho(A)=\mathbf{E}(\varrho(A))$ for all $A\in\mathcal{B}(E)$.
Note that this is the \emph{barycenter} of $\mathrm{Law}(\varrho)\in
\mathcal{P}(\mathcal{P}(E))$.

\begin{lem}
\label{lem:bary}
Let $\{\varrho_i:i\in I\}$ be a family of $\mathcal{P}(E)$-valued random 
variables on $(\Omega,\mathcal{F},\mathbf{P})$.  Then this family is tight 
if and only if the family of probability measures 
$\{\mathbf{E}\varrho_i:i\in I\}\subset\mathcal{P}(E)$ is tight.
\end{lem}

\subsection{Tightness in product spaces}

The following elementary lemma will be used repeatedly.

\begin{lem}
\label{lem:marg}
Let $\{\Xi_i:i\in I\}$ be a family of probability measures on 
$E\times\tilde E$, where $E,\tilde E$ are Polish.  Then this family is 
tight iff its marginals $\{\Xi_i(\cdot\times\tilde E):i\in I\}$ and 
$\{\Xi_i(E\times\cdot):i\in I\}$ are tight.
\end{lem}

The proof is straightforward and follows along the lines of 
\cite[lemma 1.4.3]{VW96}.

\section{Proofs}
\label{app:proofs}

This appendix contains the proofs that were omitted from the main text.

\subsection{Proof of Proposition \ref{prop:mfe}}
\label{app:prop:mfe}

Note that $\pi_{k-1}$ is a function of $Y_0,\ldots,Y_{k-1}$ only.  
Therefore 
$$
        \mathbf{E}(f(X_k,\pi_{k})|X_0,\ldots,X_k,Y_0,\ldots,Y_{k-1}) =
        \int f(X_k,\mathsf{U}(y,\pi_{k-1}P))\,\Upsilon(X_k,y)\,
	\varphi(dy),
$$
where we have used the hidden Markov property and $\pi_k = 
\mathsf{U}(Y_k,\pi_{k-1}P)$.  Using the Markov property of $(X_k)_{k\ge 
0}$ and the tower property of the conditional expectation, we obtain
$$
        \mathbf{E}(f(X_k,\pi_{k})|X_0,\ldots,X_{k-1},Y_0,\ldots,Y_{k-1}) =
        \int f(x',\mathsf{U}(y,\pi_{k-1}P))\,\Upsilon(x',y)\,
	\varphi(dy)\,P(X_{k-1},dx').
$$
As $\sigma\{X_0,\ldots,X_{k-1},\pi_0,\ldots,\pi_{k-1})\subset
\sigma\{X_0,\ldots,X_{k-1},Y_0,\ldots,Y_{k-1}\}$, the expression for
$\mathsf{\Pi}$ follows immediately.  The expression for the initial 
measure $M$ follows along similar lines.

\emph{Ergodic property:}  We begin by proving existence of the invariant 
measure.  Consider a copy $(\tilde X_k,\tilde Y_k)_{k\ge 0}$ of the hidden 
Markov model started at the stationary distribution $\tilde X_0\sim\lambda$.
Using stationarity, the process can be extended to negative times
$(\tilde X_k,\tilde Y_k)_{k\in\mathbb{Z}}$ also.  Now consider the 
measure-valued process $(\tilde X_k,\mathbf{P}(\tilde 
X_k\in\cdot|\tilde Y_k,\tilde Y_{k-1},\ldots))$ (the regular conditional 
probability always exists in a Polish state space).  It is easily seen 
that this is a stationary Markov process with transition kernel 
$\mathsf{\Pi}$.  Thus the law of $(\tilde X_0,\mathbf{P}(\tilde 
X_0\in\cdot|\tilde Y_0,\tilde Y_{-1},\ldots))$ is an invariant measure for 
$\mathsf{\Pi}$.

It remains to establish uniqueness of the invariant measure.
Endow the Polish space $E\times\mathcal{P}(E)$ with the Polish metric 
$D((x,\nu),(x',\nu'))= d(x,x')+\|\nu-\nu'\|_{\rm BL}$, where $d$ is a 
Polish metric on $E$.  In lemma \ref{lem:harris} below, it is shown that 
assumption \ref{aspt:harris} implies that
$$
        \left|\int F(z,\alpha)\,\mathsf{\Pi}^j(x,\nu,dz,d\alpha) -
        \int F(z,\alpha)\,\mathsf{\Pi}^j(x,\nu',dz,d\alpha)\right|
        \xrightarrow{j\to\infty}0
$$
whenever $F$ is $D$-Lipschitz.  Let $\Lambda$ and $\Lambda'$ be two 
$\mathsf{\Pi}$-invariant measures.   Then the marginals of $\Lambda$ and 
$\Lambda'$ on the signal state space are invariant measures for $P$.
But assumption \ref{aspt:harris} implies that $\lambda$ is the unique 
invariant measure for the signal, so we must have
$\Lambda(A\times\mathcal{P}(E))=\Lambda'(A\times\mathcal{P}(E))=\lambda(A)$.
By the Polish assumption, we therefore have the disintegrations
$$
        \Lambda(A\times B) = 
        \int I_A(x)\,I_B(\nu)\,\Lambda_{x}(d\nu)\,\lambda(dx),\qquad
        \quad
        \Lambda'(A\times B) = 
        \int I_A(x)\,I_B(\nu)\,\Lambda_{x}'(d\nu)\,\lambda(dx).
$$
It follows that
\begin{multline*}
        \left|\int F\,d\Lambda -
        \int F\,d\Lambda'\right|
        = 
        \left|\int \mathsf{\Pi}^jF(x,\nu)\,\Lambda(dx,d\nu)-
        \int \mathsf{\Pi}^jF(x,\nu)\,\Lambda'(dx,d\nu)\right| \\
        \le 
        \int |\mathsf{\Pi}^jF(x,\nu)-\mathsf{\Pi}^jF(x,\nu')|
        \,\Lambda_x(d\nu)\,\Lambda_x'(d\nu')\,\lambda(dx)
        \xrightarrow{j\to\infty}0
\end{multline*}
whenever $F$ is uniformly bounded and $D$-Lipschitz.  But this class of 
functions is measure determining, so $\Lambda$ and $\Lambda'$ must 
coincide.  The proof is complete.
\qed

\begin{lem}
\label{lem:harris}
Let $D((x,\nu),(x',\nu'))= d(x,x')+\|\nu-\nu'\|_{\rm BL}$, where $d$ 
is a Polish metric on $E$.  Then
$$
        \left|\int F(z,\alpha)\,\mathsf{\Pi}^j(x,\nu,dz,d\alpha) -
        \int F(z,\alpha)\,\mathsf{\Pi}^j(x,\nu',dz,d\alpha)\right|
        \xrightarrow{j\to\infty}0
$$
whenever $F$ is $D$-Lipschitz, provided assumptions \ref{aspt:nondeg} and 
\ref{aspt:harris} hold.
\end{lem}

\begin{proof}
Consider a copy $(\hat X_k,\hat Y_k)_{k\ge 0}$ of the hidden Markov model 
started at the initial measure $\hat X_0\sim P(x,\cdot)$, and define 
recursively $\hat\pi_k=\mathsf{U}(\hat Y_k,\hat\pi_{k-1}P)$ and 
$\hat\pi_k'=\mathsf{U}(\hat Y_k,\hat\pi_{k-1}'P)$, $k\ge 1$ with 
$\hat\pi_0=\mathsf{U}(\hat Y_0,\nu P)$ and $\hat\pi_0'= \mathsf{U}(\hat 
Y_0,\nu'P)$.  Then for $j\ge 1$, the measure 
$\mathsf{\Pi}^j(x,\nu,dz,d\alpha)$ coincides with the law of $(\hat 
X_{j-1},\hat\pi_{j-1})$, and similarly $\mathsf{\Pi}^j(x,\nu',dz,d\alpha)$ 
coincides with the law of $(\hat X_{j-1},\hat\pi_{j-1}')$.  Thus
\begin{multline*}
        \left|\int F(z,\alpha)\,\mathsf{\Pi}^j(x,\nu,dz,d\alpha) -
        \int F(z,\alpha)\,\mathsf{\Pi}^j(x,\nu',dz,d\alpha)\right|
	= |\mathbf{E}(F(\hat X_{j-1},\hat\pi_{j-1})-
	F(\hat X_{j-1},\hat\pi_{j-1}'))|
	\\
	\le
	\|F\|_L\,\mathbf{E}(\|\hat\pi_{j-1}-\hat\pi_{j-1}'\|_{\rm BL})
	\le
	\|F\|_L\,\mathbf{E}(\|\hat\pi_{j-1}-\hat\pi_{j-1}'\|_{\rm TV}).
\end{multline*}
But assumptions \ref{aspt:nondeg} and \ref{aspt:harris} allow us to apply 
the filter stability result \cite[corollary 5.5]{Van08}, which implies 
that $\mathbf{E}(\|\hat\pi_{j-1}-\hat\pi_{j-1}'\|_{\rm TV})\to 0$ as 
$j\to\infty$.  This completes the proof.
\end{proof}

\subsection{Proof of Theorem \ref{thm:approx}}
\label{app:thm:approx}

The proof of theorem \ref{thm:approx} proceeds in several steps.  
\emph{Throughout this section (appendix \ref{app:thm:approx}) we always 
presume that the assumptions of theorem \ref{thm:approx} are in force.}

We begin by proving that the convergence holds on every finite time horizon.

\begin{lem}
\label{lem:finhorz0}
$\mathbf{E}(\{\pi_{k}^{N}(f)-\pi_{k}(f)\}^2)\xrightarrow{N\to\infty}0$
for any $k<\infty$ and bounded continuous $f:E\to\mathbb{R}$.
\end{lem}

\begin{proof}
As $\mathcal{G}$ is independent of $(X_k,Y_k)_{k\ge 0}$, we can 
write $\pi_k(f)=\mathbf{E}(f(X_k)|\mathcal{F}_k^Y\vee\mathcal{G})$.  
Therefore
$$
	\mathbf{E}(\{\pi_{k}^{N}(f)-\pi_{k}(f)\}^2) =
	\mathbf{E}(\pi_{k}^{N}(f)^2 - 2\,f(X_k)\,\pi_{k}^{N}(f)) +
	\mathbf{E}(\pi_k(f)^2),
$$
where we have used assumption \ref{aspt:approx} and the tower property of 
the conditional expectation.  Define the bounded continuous function
$F:E\times\mathcal{P}(E)\to\mathbb{R}$ as $F(x,\nu) =
\nu(f)^2-2\,f(x)\,\nu(f)$.  By lemma \ref{lem:iter} and the first 
condition of theorem \ref{thm:approx}, we have $M_N\mathsf{\Pi}_N^kF\to 
M\mathsf\Pi^kF$ as $N\to\infty$.  Therefore
$$
	\mathbf{E}(\pi_{k}^{N}(f)^2 - 2\,f(X_k)\,\pi_{k}^{N}(f)) 
	\xrightarrow{N\to\infty}
	\mathbf{E}(\pi_{k}(f)^2 - 2\,f(X_k)\,\pi_{k}(f)) =
	-\mathbf{E}(\pi_{k}(f)^2).
$$
Substituting in the above expression completes the proof.
\end{proof}

We now strengthen this lemma to prove $\|\cdot\|_{\rm BL}$-convergence.

\begin{lem}
\label{lem:finhorz}
$\mathbf{E}(\|\pi_{k}^{N}-\pi_{k}\|_{\rm BL})\xrightarrow{N\to\infty}0$
for any $k<\infty$.
\end{lem}

\begin{rem}
The quantity $\mathbf{E}(\|\pi_{k}^{N}-\pi_{k}\|_{\rm BL})$ is well 
defined, as $\|\pi_{k}^{N}-\pi_{k}\|_{\rm BL}$ is measurable by
\cite[corollary A.2]{Van09}.  We will therefore employ such expressions in 
the following without further comment.
\end{rem}

\begin{proof}
Fix $\varepsilon>0$.  As $X_k$ takes values in the Polish space $E$, there 
exists a compact subset $\tilde K\subset E$ such that $\mathbf{P}(X_k\in 
\tilde K)>1-\varepsilon$.  Moreover, by the Arzel{\`a}-Ascoli theorem, 
there is an $m<\infty$ and $f_1,\ldots,f_m\in\mathrm{Lip}(E)$ such that
$\min_{\ell}\sup_{x\in\tilde K}|f(x)-f_\ell(x)|<\varepsilon$ whenever 
$f\in\mathrm{Lip}(E)$.  Define the open set $K=\{x\in 
E:d(x,K)<\varepsilon\}$.  Then $\min_{\ell}\sup_{x\in 
K}|f(x)-f_\ell(x)|<3\varepsilon$ for any
$f\in\mathrm{Lip}(E)$, so
\begin{equation*}
\begin{split}
	& \mathbf{E}(\|\pi_{k}^{N}-\pi_{k}\|_{\rm BL}) \le
	\mathbf{E}\Bigg(\sup_{f\in\mathrm{Lip}(E)}
		|\pi_{k}^{N}(fI_K)-\pi_{k}(fI_K)|\Bigg)
	+ 
	\mathbf{E}\Bigg(\sup_{f\in\mathrm{Lip}(E)}
		|\pi_{k}^{N}(fI_{K^c})-\pi_{k}(fI_{K^c})|\Bigg) \\
	&\qquad\le
	6\varepsilon +
	\mathbf{E}\Big(\max_{\ell}
		|\pi_{k}^{N}(f_\ell I_K)-\pi_{k}(f_\ell I_K)|\Big)
	+ \mathbf{E}(\pi_{k}^{N}(K^c)) + \mathbf{P}(X_k\in K^c) \\
	&\qquad\le
	6\varepsilon +
	\mathbf{E}\Big(\max_{\ell}
		|\pi_{k}^{N}(f_\ell)-\pi_{k}(f_\ell)|\Big)
	+ 2\,\mathbf{E}(\pi_{k}^{N}(K^c)) + 2\,\mathbf{P}(X_k\in K^c) \\
	&\qquad\le
	6\varepsilon +
	\sum_{\ell=1}^m
	\sqrt{\mathbf{E}(\{\pi_{k}^{N}(f_\ell)-\pi_{k}(f_\ell)\}^2)}
	+ 2\,\mathbf{E}(\pi_{k}^{N}(K^c)) + 2\,\mathbf{P}(X_k\in K^c).
\end{split}
\end{equation*}
As $K^c$ is closed and $M_N\mathsf\Pi_N^k\Rightarrow M\mathsf\Pi^k$
by lemma \ref{lem:iter} and the first condition of theorem 
\ref{thm:approx}, applying the Portmanteau theorem to the second term
and lemma \ref{lem:finhorz0} to the first term gives
$$
	\limsup_{N\to\infty}\mathbf{E}(\|\pi_{k}^{N}-\pi_{k}\|_{\rm BL}) 
	\le
	6\varepsilon + 4\,\mathbf{P}(X_k\in K^c) \le
	6\varepsilon + 4\,\mathbf{P}(X_k\in \tilde K^c) \le
	10\varepsilon.
$$
But $\varepsilon>0$ was arbitrary, so the proof is complete.
\end{proof}

We have now established convergence of the filters as $N\to\infty$ for a 
fixed time $k$.  The idea is now to repeat the proofs for the case where 
we let the number of particles and time go to infinity simultaneously. We 
will repeat almost identically the steps used in the last two lemmas, 
where the finite time weak convergence $M_N\mathsf\Pi_N^k\Rightarrow 
M\mathsf\Pi^k$ used in the proofs is replaced by the following ergodic 
lemma (recall that $\Lambda$ is the unique invariant measure of 
$\mathsf\Pi$).

\begin{lem}
\label{lem:erg}
For any sequence $T_N\nearrow\infty$ as $N\to\infty$, define
the probability measures
$$
	\int F(x,\nu)\,\Lambda_N(dx,d\nu) :=
	\mathbf{E}\left[
	\frac{1}{T_N}\sum_{k=1}^{T_N} F(X_k,\pi_k^N)
	\right],
$$
for every $N\in\mathbb{N}$.  Then $\Lambda_N\Rightarrow\Lambda$ as 
$N\to\infty$.
\end{lem}

\begin{proof}
We first show that the family $\{\Lambda_N:N\in\mathbb{N}\}$ is
tight.  It suffices to show that the marginals are tight by lemma 
\ref{lem:marg}.  But the first marginal of $\Lambda_N$ is 
$T_N^{-1}\sum_{k=1}^{T_N}\mu P^k$, which converges to the signal invariant 
measure $\lambda$ by assumption \ref{aspt:harris}.  This establishes 
tightness of the first marginal.  By lemma \ref{lem:bary}, tightness of 
the second marginal follows from the second condition of theorem 
\ref{thm:approx}.

Having established tightness, it remains to show that every convergent 
subsequence of $\{\Lambda_N:N\in\mathbb{N}\}$ converges to $\Lambda$.  In 
fact, it suffices to show that the limit of every convergent subsequence 
must be an invariant measure of $\mathsf\Pi$, as the latter is unique by 
proposition \ref{prop:mfe}.

Let $\Lambda_{Q(N)}$ be a weakly convergent subsequence of 
$\{\Lambda_N:N\in\mathbb{N}\}$ and denote its limit as $\tilde\Lambda$.
By the first condition of theorem \ref{thm:approx} and lemma 
\ref{lem:iter}, we have $\Lambda_{Q(N)}\mathsf\Pi_{Q(N)}\Rightarrow
\tilde\Lambda\mathsf\Pi$.  But note that
$$
	\Lambda_{Q(N)}\mathsf\Pi_{Q(N)}
	= 
	\frac{1}{T_{Q(N)}}\sum_{k=1}^{T_{Q(N)}} 
	M_{Q(N)}\mathsf\Pi_{Q(N)}^{k+1}
	=
	\Lambda_{Q(N)}
	+ \frac{1}{T_{Q(N)}} \{M_{Q(N)}\mathsf\Pi_{Q(N)}^{T_{Q(N)}+1}
	- M_{Q(N)}\mathsf\Pi_{Q(N)}\}.
$$
We therefore have
$$
	\|\tilde\Lambda-\tilde\Lambda\mathsf\Pi\|_{\rm BL} =
	\lim_{N\to\infty}
	\|\Lambda_{Q(N)}-\Lambda_{Q(N)}\mathsf\Pi_{Q(N)}\|_{\rm BL}
	\le
	\limsup_{N\to\infty}\,\frac{2}{T_{Q(N)}}=0,
$$
so $\tilde\Lambda$ is an invariant measure for $\mathsf\Pi$.  
\end{proof}

We now repeat the arguments of lemmas \ref{lem:finhorz0} and 
\ref{lem:finhorz} with the necessary modifications.

\begin{lem}
\label{lem:infhorz0}
For any sequence $T_N\nearrow\infty$ as $N\to\infty$,
$$
	\mathbf{E}\left[
	\frac{1}{T_N}\sum_{k=1}^{T_N}\{\pi_{k}^{N}(f)-\pi_{k}(f)\}^2
	\right]
	\xrightarrow{N\to\infty}0
$$
for any bounded continuous function $f:E\to\mathbb{R}$.
\end{lem}

\begin{proof}
As in the proof of lemma \ref{lem:finhorz0}, we can write
$$
	\mathbf{E}\left[
	\frac{1}{T_N}\sum_{k=1}^{T_N}\{\pi_{k}^{N}(f)-\pi_{k}(f)\}^2
	\right]
	=
	\mathbf{E}\left[
	\frac{1}{T_N}\sum_{k=1}^{T_N}F(X_k,\pi_k^N)
	\right] +
	\mathbf{E}\left[
	\frac{1}{T_N}\sum_{k=1}^{T_N} \pi_k(f)^2
	\right],
$$
where $F(x,\nu) = \nu(f)^2-2\,f(x)\,\nu(f)$.  By lemma \ref{lem:erg}
$$
	\mathbf{E}\left[
	\frac{1}{T_N}\sum_{k=1}^{T_N}F(X_k,\pi_k^N)
	\right] \xrightarrow{N\to\infty}
	-\,\mathbf{E}(\mathbf{E}(f(\tilde X_0)|\tilde Y_0,
		\tilde Y_{-1},\ldots)^2),
$$
where we have used the expression for $\Lambda$ in terms of the 
stationary copy $(\tilde X_k,\tilde Y_k)_{k\in\mathbb{Z}}$ given in 
the proof of proposition \ref{prop:mfe}.  The proof would evidently be 
complete if we can show that
$$
	\limsup_{k\to\infty}\mathbf{E}(\pi_k(f)^2) \le
	\mathbf{E}(\mathbf{E}(f(\tilde X_0)|\tilde Y_0,
		\tilde Y_{-1},\ldots)^2).
$$
To this end, we proceed as follows.  First, note that
$$
	\mathbf{E}(\pi_{k+\ell}(f)^2) =
	\mathbf{E}(\mathbf{E}(f(X_{k+\ell})|Y_0,\ldots,Y_{k+\ell})^2)
	\le
	\mathbf{E}(\mathbf{E}(f(X_{k+\ell})|X_0,\ldots,X_{\ell},
		Y_0,\ldots,Y_{k+\ell})^2),
$$
where we have used the tower property of the conditional expectation and 
Jensen's inequality.  But by the Markov property of $(X_k,Y_k)_{k\ge 0}$,
we can write
$$
	\mathbf{E}(f(X_{k+\ell})|X_0,\ldots,X_{\ell},
		Y_0,\ldots,Y_{k+\ell}) =
	\mathbf{E}(f(X_{k+\ell})|X_{\ell},
		Y_\ell,\ldots,Y_{k+\ell}) :=
	G_k(X_\ell,Y_\ell,\ldots,Y_{k+\ell}),
$$
where the function $G_k$ does not depend on $\ell$.
Using assumption \ref{aspt:harris}, it follows easily that
$$
	\limsup_{\ell\to\infty}\mathbf{E}(\pi_{\ell}(f)^2)=
	\limsup_{\ell\to\infty}\mathbf{E}(\pi_{k+\ell}(f)^2) 
	\le
	\mathbf{E}(G_k(\tilde X_{-k},\tilde Y_{-k},\ldots,\tilde Y_0)^2).
$$
But $G_k(\tilde X_{-k},\tilde Y_{-k},\ldots,\tilde Y_0) =
\mathbf{E}(f(\tilde X_{0})|\tilde Y_0,\ldots,\tilde Y_{-k},\tilde X_{-k})$,
so by the Markov property of $(\tilde X_k,\tilde Y_k)_{k\ge 0}$
$$
	\limsup_{\ell\to\infty}\mathbf{E}(\pi_{\ell}(f)^2) \le
	\mathbf{E}(G_k(\tilde X_{-k},\tilde Y_{-k},\ldots,\tilde Y_0)^2) =
	\mathbf{E}(\mathbf{E}(f(\tilde X_{0})|
	\sigma\{\tilde Y_\ell:\ell\le 0\}
	\vee\sigma\{\tilde X_\ell:\ell\le -k\})^2)
$$
for all $k$.  Letting $k\to\infty$ in this expression and using that
$$
	\bigcap_{k\ge 0}\sigma\{\tilde Y_\ell:\ell\le 0\}
        \vee\sigma\{\tilde X_\ell:\ell\le -k\} =
	\sigma\{\tilde Y_\ell:\ell\le 0\}\quad
	\mathbf{P}\mbox{-a.s.}
$$
by \cite[theorem 4.2]{Van08} (which holds by virtue of
assumptions \ref{aspt:nondeg} and \ref{aspt:harris}), the proof is complete.
\end{proof}

\begin{lem}
\label{lem:infhorz}
For any sequence $T_N\nearrow\infty$ as $N\to\infty$,
$$
	\mathbf{E}\left[
	\frac{1}{T_N}\sum_{k=1}^{T_N}\|\pi_{k}^{N}-\pi_{k}\|_{\rm BL}
	\right]
	\xrightarrow{N\to\infty}0.
$$
\end{lem}

\begin{proof}
Fix $\varepsilon>0$, and choose a compact subset $\tilde K\subset E$ such 
that $\lambda(\tilde K)>1-\varepsilon$.  Construct 
$f_1,\ldots,f_m\in\mathrm{Lip}(E)$ and $K$ as in the proof of lemma 
\ref{lem:finhorz}.  Then we can estimate
\begin{multline*}
	\mathbf{E}\left[
	\frac{1}{T_N}\sum_{k=1}^{T_N}\|\pi_{k}^{N}-\pi_{k}\|_{\rm BL}
	\right]
	\le
	6\varepsilon + 
	\sum_{\ell=1}^m\mathbf{E}\left[
	\frac{1}{T_N}\sum_{k=1}^{T_N}
	\{\pi_{k}^{N}(f_\ell)-\pi_{k}(f_\ell)\}^2
	\right]^{1/2} \\
	+
	2\,\mathbf{E}\left[\frac{1}{T_N}\sum_{k=1}^{T_N}
	\pi_{k}^{N}(K^c)\right]
	+
	\frac{1}{T_N}\sum_{k=1}^{T_N} 2\,\mathbf{P}(X_k\in K^c).
\end{multline*}
Applying lemma \ref{lem:infhorz0} to the first term,
lemma \ref{lem:erg} and the Portmanteau theorem to the second term, and
assumption \ref{aspt:harris} to the third term, we find that
$$
	\limsup_{N\to\infty}
	\mathbf{E}\left[
	\frac{1}{T_N}\sum_{k=1}^{T_N}\|\pi_{k}^{N}-\pi_{k}\|_{\rm BL}
	\right] 
	\le
	6\varepsilon + 4\,\lambda(K^c) \le
	6\varepsilon + 4\,\lambda(\tilde K^c) \le
	10\varepsilon.
$$
But $\varepsilon>0$ was arbitrary, so the proof is complete.
\end{proof}

We can now complete the proof of theorem \ref{thm:approx}.

\begin{proof}[Proof of Theorem \ref{thm:approx}]
Suppose that
$$
        \limsup_{N\to\infty}\,\sup_{T\ge 0}\,
        \mathbf{E}\left[
        \frac{1}{T}\sum_{k=1}^{T}\|\pi^{N}_{k}-\pi_{k}\|_{\rm BL}
        \right]=\varepsilon>0.
$$
Then we can find subsequences $Q(N)\nearrow\infty$ and $T_{Q(N)}$ such 
that
$$
        \mathbf{E}\left[
        \frac{1}{T_{Q(N)}}\sum_{k=1}^{T_{Q(N)}}
	\|\pi^{Q(N)}_{k}-\pi_{k}\|_{\rm BL}
        \right]>\frac{\varepsilon}{2}
	\quad\mbox{for all }N.
$$
Suppose first that $T_{Q(N)}\le T_{\rm max}$ is a bounded sequence.  Then 
lemma \ref{lem:finhorz} gives
$$
        \mathbf{E}\left[
        \frac{1}{T_{Q(N)}}\sum_{k=1}^{T_{Q(N)}}
	\|\pi^{Q(N)}_{k}-\pi_{k}\|_{\rm BL}
        \right] \le
	\max_{k\le T_{\rm max}}
	\mathbf{E}(\|\pi^{Q(N)}_{k}-\pi_{k}\|_{\rm BL})
	\xrightarrow{N\to\infty}0,
$$
so we have a contradiction.  But if $T_{Q(N)}$ is an unbounded sequence,
we can find a further subsequence $R(N)\nearrow\infty$ such that
$T_{R(N)}\nearrow\infty$, and by lemma \ref{lem:infhorz}
$$
        \mathbf{E}\left[
        \frac{1}{T_{R(N)}}\sum_{k=1}^{T_{R(N)}}
	\|\pi^{R(N)}_{k}-\pi_{k}\|_{\rm BL}
        \right] \xrightarrow{N\to\infty}0
$$
which is again a contradiction.  The proof is complete.
\end{proof}

\subsection{Proof of Corollary \ref{cor:esterr}}
\label{app:cor:esterr}
Note that we can estimate
\begin{multline*}
	\mathbf{E}\left[\left|
        \frac{1}{T}\sum_{k=1}^{T}
	\left(f(X_k)-\int f\,d\pi_k^N\right)^2 -
        \frac{1}{T}\sum_{k=1}^{T}
	\left(f(X_k) -
	\int f\,d\pi_k\right)^2\right|
	\right]
	\\
	\le 
	4\,\|f\|_\infty\,
	\mathbf{E}\left[
        \frac{1}{T}\sum_{k=1}^{T}
	|\pi_k^N(f)-\pi_k(f)|
	\right]
	\le 
	4\,\|f\|_\infty\,
	\mathbf{E}\left[
        \frac{1}{T}\sum_{k=1}^{T}
	\{\pi_k^N(f)-\pi_k(f)\}^2
	\right]^{1/2}.
\end{multline*}
The result is now easily obtained by following the 
same steps as in the proof of theorem 
\ref{thm:approx}.
\qed

\subsection{Proof of Proposition \ref{prop:bootconv}}
\label{app:prop:bootconv}

We begin by proving a general continuity result for $\mathsf{R}_N$.

\begin{lem}
\label{lem:resample}
$\mathsf{R}_N(\nu_N,\cdot)\Rightarrow\delta_\nu$ as $N\to\infty$
whenever $\nu_N\Rightarrow\nu$ as $N\to\infty$.
\end{lem}

\begin{proof}
It follows immediately from the definition that the barycenter of 
$\mathsf{R}_N(\rho,\cdot)$ is $\rho$ for any $\rho\in\mathcal{P}(E)$.  
Therefore, by lemma \ref{lem:bary}, the sequence 
$\{\mathsf{R}_N(\nu_N,\cdot):N\in\mathbb{N}\}$ is tight.  It thus suffices 
to prove that every convergent subsequence converges to $\delta_\nu$.
Let $Q(N)$ be any subsequence such that
$\mathsf{R}_{Q(N)}(\nu_{Q(N)},\cdot)\Rightarrow R$ for some 
$R\in\mathcal{P}(\mathcal{P}(E))$.  Note that for any probability measure 
$\rho$
$$
        \int |\rho'(f)-\rho(f)|\,\mathsf{R}_{N}(\rho,d\rho') \le
        \left[
        \int \{\rho'(f)-\rho(f)\}^2\,
        \mathsf{R}_{N}(\rho,d\rho')\right]^{\frac{1}{2}} =
        \sqrt{\frac{\rho(f^2)-\rho(f)^2}{N}} \le 
	\frac{\|f\|_\infty^2}{\sqrt{N}}.
$$
In particular, this shows that
\begin{multline*}
	\int |\nu'(f)-\nu(f)|\,R(d\nu') =
	\lim_{N\to\infty}\int |\nu'(f)-\nu(f)|\,
	\mathsf{R}_{Q(N)}(\nu_{Q(N)},d\nu') \\
	\le
	\lim_{N\to\infty}\int |\nu'(f)-\nu_{Q(N)}(f)|\,
        \mathsf{R}_{Q(N)}(\nu_{Q(N)},d\nu') +
	\lim_{N\to\infty}|\nu_{Q(N)}(f)-\nu(f)| = 0
\end{multline*}
for any bounded continuous function $f:E\to\mathbb{R}$.  Thus we must have
$R=\delta_\nu$.
\end{proof}

We can now complete the proof.

\begin{proof}[Proof of Proposition \ref{prop:bootconv}]
As $\Upsilon(\cdot,y)$ is bounded and continuous 
(assumption \ref{aspt:feller}), we have
$$
	\int f(x)\,\Upsilon(x,y)\,\nu_n(dx) \to
	\int f(x)\,\Upsilon(x,y)\,\nu(dx)
$$
for every $y$ whenever $f:E\to\mathbb{R}$ is bounded and continuous and 
$\nu_n\Rightarrow\nu$.  This implies that 
$\mathsf{U}(y,\nu_n)\Rightarrow\mathsf{U}(y,\nu)$ for every $y$, so in 
particular $(x',\nu')\mapsto F(x',\mathsf{U}(y,\nu'))\,\Upsilon(x',y)$ is 
bounded and continuous for every $y$ whenever $F:E\times\mathcal{P}(E)\to
\mathbb{R}$ is a bounded continuous function.  Using the Feller property 
of $P$ and lemma \ref{lem:resample}, it follows that whenever $x_N\to x$ 
and $\nu_N\Rightarrow\nu$
$$
	\int F(x',\mathsf{U}(y,\nu'))\,
	\Upsilon(x',y)\,\mathsf{R}_N(\nu_NP,d\nu')\,
	P(x_N,dx') \xrightarrow{N\to\infty}
	\int F(x',\mathsf{U}(y,\nu P))\,\Upsilon(x',y)\,
	P(x,dx')
$$
for every $y$ and bounded continuous function $F$.  But then we obtain
by dominated convergence
\begin{multline*}
        \int F(x',\nu')\,\mathsf{\Pi}_{N}(x_N,\nu_N,dx',d\nu') =
	\int F(x',\mathsf{U}(y,\nu'))\,
	\Upsilon(x',y)\,\mathsf{R}_N(\nu_NP,d\nu')\,
	P(x_N,dx')\,\varphi(dy) \\
	\xrightarrow{N\to\infty}
	\int F(x',\mathsf{U}(y,\nu P))\,\Upsilon(x',y)\,
	P(x,dx')\,\varphi(dy) = 
        \int F(x',\nu')\,\mathsf{\Pi}(x,\nu,dx',d\nu').
\end{multline*}
It remains to show that $M_N\Rightarrow M$.  This follows immediately, 
however, from lemma \ref{lem:resample}, the fact that 
$\pi\mapsto\mathsf{U}(y,\pi)$ is continuous, and dominated convergence.
The finite time convergence now follows from lemma \ref{lem:finhorz} 
(which does not rely on assumption \ref{aspt:harris}), and the proof is 
complete.
\end{proof}

\subsection{Proof of Theorem \ref{thm:main}}
\label{app:thm:main}

As both assumptions require geometric ergodicity, we fix throughout the 
corresponding function $V$ (which, by definition, is presumed to have 
compact level sets).  To complete the proof, it only remains to prove the 
tightness assumption of corollary \ref{cor:bootconv}.  We will in fact 
verify the simpler sufficient condition in lemma \ref{lem:tight} through 
the following elementary result.

\begin{lem}
Suppose that $\sup_{k,N}\mathbf{E}\pi_{k}^{N}(V)<\infty$.
Then the tightness assumption holds.
\end{lem}

\begin{proof}
The level sets $C_r = \{x\in E:V(x)\le r\}$ are compact.  But as
$$
        \sup_{k,N}\mathbf{E}\pi_{k}^N(C_r^c) =
        \sup_{k,N}\mathbf{E}\pi_{k}^N(V>r) \le
        \frac{\sup_{k,N}\mathbf{E}\pi_{k}^N(V)}{r}
        \xrightarrow{r\to\infty}0,
$$
evidently the family $\{\mathbf{E}\pi_{k}^N:k,N\ge 1\}$ is tight, and we 
may invoke lemma \ref{lem:tight}.
\end{proof}

In the following, it is convenient to introduce the measure-valued process
$$
	\pi_{k-}^N(A) := \frac{\int I_A(x)\,\Upsilon(x,Y_k)^{-1}\,
		\pi_k^N(dx)}{\int \Upsilon(x,Y_k)^{-1}\,\pi_k^N(dx)},
$$
so that $\pi_k^N = \mathsf{U}(Y_k,\pi_{k-}^N)$.  Note that $\pi_{k-}^N$ is 
the bootstrap particle filter approximation to the one step predictor 
$\pi_{k-}$ (in fact, our main results are easily adapted to establish 
uniform time average convergence of $\pi_{k-}^N$ to $\pi_{k-}$).  The 
following result is the key tool that allows us to establish tightness.  
The condition of this lemma---essentially, the requirement that the update 
step $\pi_{k-}^N\mapsto\pi_k^N$ does not `expand' too much---will be 
verified separately under the assumptions \ref{aspt:casei} and 
\ref{aspt:caseii}.

\begin{lem}
Suppose the signal is geometrically ergodic and $\mu(V)<\infty$.
If there exist constants $c_1,c_2\ge 0$ such that $\mathbf{E}\pi_{k}^N(V)\le 
c_1\,\mathbf{E}\pi_{k-}^N(V)+c_2$ for all $k$, then 
$\sup_{k,N}\mathbf{E}\pi_{k}^{N}(V)<\infty$.
\end{lem}

\begin{proof}
Note that $\mathbf{E}\pi_{k-}(f)=\mu(P^{k}f)=\mathbf{E}\pi_{0-}^N(P^kf)$
for all $|f|\le V$.  Therefore
$$
        \mathbf{E}\pi_{k-}^N(f)-\mathbf{E}\pi_{k-}(f) =
        \sum_{\ell=1}^k \{\mathbf{E}\pi_{\ell-}^N(P^{k-\ell}f)-
                \mathbf{E}\pi_{(\ell-1)-}^N(P^{k-\ell+1}f)\}.
$$
But note that $\mathbf{E}\pi_{\ell-}^N(f)=\mathbf{E}\pi_{\ell-1}^N(Pf)$, 
as we may average over the last sampling step.  Therefore
$$
        \mathbf{E}\pi_{k-}^N(f)-\mathbf{E}\pi_{k-}(f) =
        \sum_{\ell=1}^k \{\mathbf{E}\pi_{\ell-1}^N(P^{k-\ell+1}f)-
                \mathbf{E}\pi_{(\ell-1)-}^N(P^{k-\ell+1}f)\}.
$$
As the signal is assumed geometrically ergodic, we have 
$\lambda(V)<\infty$ and
$$
        \|P^k(x,\,\cdot\,)-\lambda\|_V\le c_3\,V(x)\,\beta^k
        \quad\mbox{for all }x\in E, ~k\ge 0,
$$
for some constants $c_3<\infty$, $\beta<1$. In particular, we 
find that for any measures $\nu_1,\nu_2$
$$
        \|\nu_1P^k-\nu_2P^k\|_V=
        \sup_{|f|\le V}|\{\nu_1-\nu_2\}(P^kf-\lambda(f))|\le
        c_3\,\beta^k\,|\nu_1-\nu_2|(V) = c_3\,\beta^k\,\|\nu_1-\nu_2\|_V.
$$
Therefore we can estimate
$$
        \|\mathbf{E}\pi_{k-}^N-\mathbf{E}\pi_{k-}\|_V\le
        \sum_{\ell=1}^k \|\mathbf{E}\pi_{\ell-1}^NP^{k-\ell+1}-
                \mathbf{E}\pi_{(\ell-1)-}^NP^{k-\ell+1}\|_V \le
        \sum_{\ell=1}^k c_3\,\beta^{k-\ell+1}\,
        \|\mathbf{E}\pi_{\ell-1}^N-\mathbf{E}\pi_{(\ell-1)-}^N\|_V.
$$
In particular, we find that
$$
        \mathbf{E}\pi_{k-}^N(V)\le
        \mathbf{E}\pi_{k-}(V) + \|\mathbf{E}\pi_{k-}^N-\mathbf{E}\pi_{k-}\|_V
        \le
        \mu P^k(V) + 
        \sum_{\ell=1}^k c_3\,\beta^{k-\ell+1}\,
        \{\mathbf{E}\pi_{\ell-1}^N(V)+\mathbf{E}\pi_{(\ell-1)-}^N(V)\}.
$$
By the assumption of the lemma we now obtain
$$
        \mathbf{E}\pi_{k-}^N(V)\le
        \mu P^k(V) + 
        \sum_{\ell=1}^k c_3\,\beta^{k-\ell+1}\,
        \{(c_1+1)\,\mathbf{E}\pi_{(\ell-1)-}^N(V) + c_2\}.
$$
But $\mu P^k(V)\to \lambda(V)$ as $k\to\infty$, so $c_4=\sup_k\mu 
P^k(V) + c_2c_3\beta/(1-\beta)<\infty$.  By lemma \ref{lem:gron} below
$$
        \mathbf{E}\pi_{k-}^N(V) \le c_4\,\exp\left(
                \sum_{\ell=1}^k (c_1+1)c_3\,\beta^{k-\ell+1}
        \right)\le
        c_4\,\exp\left(
        \frac{\beta(c_1+1)c_3}{1-\beta}
        \right).
$$
But as $\mathbf{E}\pi_{k}^N(V)\le c_1\,\mathbf{E}\pi_{k-}^N(V)+c_2$,
the proof is evidently complete.
\end{proof}

In the previous proof, we needed the following.

\begin{lem}[Discrete Gr{\"o}nwall]
\label{lem:gron}
Suppose $(A,\alpha_k,B_k)$, $k\ge 0$ are nonnegative scalars such that
$$
        \alpha_k \le A + \sum_{\ell=1}^k B_\ell\,\alpha_{\ell-1}
        \quad\mbox{for all }k\ge 0.
$$
Then it must be the case that
$$
        \alpha_k \le A\,\exp\left(
                \sum_{\ell=1}^k B_\ell
        \right)\quad\mbox{for all }k\ge 0.
$$
\end{lem}

\begin{proof}
As $\log(1+x)\le x$, it suffices to prove the first inequality in
$$
        \alpha_k\le A\prod_{\ell=1}^{k}(1+B_\ell)=
        A\,\exp\left(
        \sum_{\ell=1}^k\log(1+B_\ell)
        \right)\le
        A\,\exp\left(\sum_{\ell=1}^k B_\ell\right).
$$
We proceed by induction.  Clearly the statement is true for $k=0$.  Now 
suppose we have verified the statement for all $\ell<k$.  Then by 
assumption
$$
        \alpha_k \le A + A\sum_{\ell=1}^k 
B_\ell\prod_{r=1}^{\ell-1}(1+B_r) =
        A + A\sum_{\ell=1}^k \left\{(1+B_\ell)\prod_{r=1}^{\ell-1}(1+B_r)-
        \prod_{r=1}^{\ell-1}(1+B_r)\right\}.
$$
But the rightmost expression is evidently a telescoping sum which reduces 
to
$$
        \alpha_k \le A + A\left\{\prod_{r=1}^{k}(1+B_r)-1\right\} =
        A\prod_{r=1}^{k}(1+B_r).
$$
The proof is complete.
\end{proof}

It remains to show that $\mathbf{E}\pi_{k}^N(V)\le
c_1\,\mathbf{E}\pi_{k-}^N(V)+c_2$.  Here we distinguish between the two 
separate cases of assumptions \ref{aspt:casei} and \ref{aspt:caseii}.
The results below complete the proof of theorem \ref{thm:main}.

\subsubsection{Case I}

In the setting of assumption \ref{aspt:casei}, the result is 
straightforward.

\begin{lem}
Suppose that assumptions \ref{aspt:nondeg} and \ref{aspt:casei} hold.
Then $\mathbf{E}\pi_{k}^N(V)\le c_1\,\mathbf{E}\pi_{k-}^N(V)+c_2$.
\end{lem}

\begin{proof}
Note that
$$
	\pi_{k}^N(V) = \frac{\int V(x)\,\Upsilon(x,Y_k)\,
	\pi_{k-}^N(dx)}{\int\Upsilon(x,Y_k)\,\pi_{k-}^N(dx)} \le
	\frac{u_+(Y_k)}{u_-(Y_k)}\int V(x)\,\pi_{k-}^N(dx).
$$
We may therefore estimate
\begin{multline*}
	\mathbf{E}(\pi_{k}^N(V)|Y_0,\ldots,Y_{k-1}) \le
	\pi_{k-}^N(V)\,\mathbf{E}\left[\left.
	\frac{u_+(Y_k)}{u_-(Y_k)}\right|Y_0,\ldots,Y_{k-1}\right] \\
	= \pi_{k-}^N(V)\int \frac{u_+(y)}{u_-(y)}\,
		\Upsilon(x,y)\,\varphi(dy)\,\pi_{k-}(dx)
	\le \pi_{k-}^N(V)\int \frac{u_+(y)^2}{u_-(y)}\,\varphi(dy).
\end{multline*}
Taking the expectation of both sides completes the proof.
\end{proof}

\subsubsection{Case II}

In the setting of assumption \ref{aspt:caseii} we will 
need the following result, whose proof we recall for the reader's 
convenience, to control the growth of the update step.

\begin{lem}[Chebyshev's covariance inequality]
\label{lem:cheb}
Let $\psi,\phi:\mathbb{R}\to\mathbb{R}$ be nondecreasing functions and let
$\nu$ be any probability measure on 
$(\mathbb{R},\mathcal{B}(\mathbb{R}))$.  Then
$$
        \int \psi(x)\,\phi(x)\,\nu(dx) - 
        \int \psi(x)\,\nu(dx)\int\phi(x)\,\nu(dx)\ge 0,
$$
i.e., the covariance of $\psi$ and $\phi$ is always nonnegative.
\end{lem}

\begin{proof}
Note that
$$
        \int \psi(x)\,\phi(x)\,\nu(dx) - 
        \int \psi(x)\,\nu(dx)\int\phi(x)\,\nu(dx) =
        \frac{1}{2}\int
        \{\psi(x)-\psi(x')\}\,\{\phi(x)-\phi(x')\}\,
        \nu(dx)\,\nu(dx').
$$
But by our assumptions the integrand is nonnegative, and the result 
follows.
\end{proof}

We obtain the following result.

\begin{lem}
Suppose that assumption \ref{aspt:caseii} holds and $\mu(V)<\infty$.
Then $\mathbf{E}\pi_{k}^N(V)\le c_1\,\mathbf{E}\pi_{k-}^N(V)+c_2$.
\end{lem}

\begin{proof}
We choose $\varphi$ to be the Lebesgue measure and 
$\Upsilon(x,y)=q_\xi(\sigma(x)^{-1}\{y-h(x)\})$.  Note that
$$
	a_1\,q(|\sigma(x)^{-1}\{y-h(x)\}|)\le \Upsilon(x,y)\le 
	a_2\,q(|\sigma(x)^{-1}\{y-h(x)\}|).
$$
As all finite dimensional norms are equivalent, we have
$\kappa^{-1}\|v\|\le |v|\le \kappa\|v\|$ for all $v$ and some $\kappa>0$.
Using $(a+b)^p\le C_p(a^p+b^p)$ where $C_p=\max(1,2^{p-1})$,
we can therefore estimate
\begin{equation*}
\begin{split}
	\|h(x)\|^p
	&\le\{\|Y_k-h(x)\|+\|Y_k\|\}^p \\
	&\le\{\varepsilon^{-1}\|\sigma(x)^{-1}\{Y_k-h(x)\}\|+\|Y_k\|\}^p
	\\
	&\le C_p\varepsilon^{-p}\kappa^p|\sigma(x)^{-1}\{Y_k-h(x)\}|^p
	+C_p\|Y_k\|^p.
\end{split}
\end{equation*}
In particular, using that $V(x)\le b_3\|h(x)\|^p+b_4$, we find
$$
	\pi_k^N(V) \le
	\frac{C_p\kappa^pa_2b_3}{\varepsilon^pa_1}
	\frac{\int |\sigma(x)^{-1}\{Y_k-h(x)\}|^p\,
		q(|\sigma(x)^{-1}\{Y_k-h(x)\}|)\,\pi_{k-}^N(dx)}{
		\int q(|\sigma(x)^{-1}\{Y_k-h(x)\}|)\,\pi_{k-}^N(dx)} 
	+ C_pb_3\|Y_k\|^p + b_4.
$$
But $q$ is nonincreasing and $v\mapsto v^p$ is nondecreasing, so by
lemma \ref{lem:cheb}
$$
	\pi_k^N(V) \le
	\frac{C_p\kappa^pa_2b_3}{\varepsilon^pa_1}
	\int |\sigma(x)^{-1}\{Y_k-h(x)\}|^p\,\pi_{k-}^N(dx)
	+ C_pb_3\|Y_k\|^p + b_4.
$$
Now note that
\begin{equation*}
\begin{split}
	|\sigma(x)^{-1}\{Y_k-h(x)\}|^p
	&\le C_p\varepsilon^{-p}\kappa^p\|Y_k\|^p
		+C_p\varepsilon^{-p}\kappa^p\|h(x)\|^p \\
	&\le C_p\varepsilon^{-p}\kappa^p\|Y_k\|^p
		+ C_p\varepsilon^{-p}\kappa^p\{V(x)-b_2\}/b_1.
\end{split}
\end{equation*}
Substituting in the above expression, we obtain
$$
	\pi_k^N(V) \le
	\frac{C_p^2\kappa^{2p}a_2b_3}{\varepsilon^{2p}a_1b_1}
	\int V(x)\,\pi_{k-}^N(dx)
	+ C_pb_3\left[1+\frac{C_p\kappa^{2p}a_2}{\varepsilon^{2p}a_1}\right] 
	\|Y_k\|^p + b_4 - 
	\frac{C_p^2\kappa^{2p}a_2b_2b_3}{\varepsilon^{2p}a_1b_1}.
$$
Finally, note that
\begin{equation*}
\begin{split}
	\mathbf{E}(\|Y_k\|^p)&\le
	C_p\{\mathbf{E}(\|h(X_k)\|^p)+\mathbf{E}(\|\sigma(X_k)\,\xi_k\|^p)\}
	\\ &\le
	C_p\{\mathbf{E}(V(X_k))/b_1-b_2/b_1+
	\varepsilon^{-p}\mathbf{E}(\|\xi_k\|^p)\},
\end{split}
\end{equation*}
which is bounded uniformly in $k$ by our assumptions.  Therefore
$$
	\mathbf{E}\pi_k^N(V) \le
	\frac{C_p^2\kappa^{2p}a_2b_3}{\varepsilon^{2p}a_1b_1}
	\,\mathbf{E}\pi_{k-}^N(V)
	+ C_pb_3\left[1+\frac{C_p\kappa^{2p}a_2}{\varepsilon^{2p}a_1}\right] 
		\sup_{k\ge 0}\mathbf{E}(|Y_k|^p)
	+ b_4 - \frac{C_p^2\kappa^{2p}a_2b_2b_3}{\varepsilon^{2p}a_1b_1},
$$
which completes the proof.
\end{proof}

\bibliographystyle{abbrv}
\bibliography{ref}

\end{document}